\documentclass[11pt]{amsart}
\usepackage[active]{srcltx}
\usepackage[utf8]{inputenc}
\usepackage{amsmath,mathtools}
\usepackage{stmaryrd}
\usepackage{MnSymbol}
\usepackage{hyperref}
\usepackage{tikz}
\usetikzlibrary{arrows,automata}
\usepackage{enumitem}
\usepackage{mathrsfs}
\usepackage{pst-plot}
\usepackage{auto-pst-pdf}
\usepackage[all]{xy}
\usepackage{stackrel}
\usepackage{listings}
\lstdefinelanguage{GAP}{%
  morekeywords={%
    Assert,Info,IsBound,QUIT,%
    TryNextMethod,Unbind,and,break,%
    continue,do,elif,%
    else,end,false,fi,for,%
    function,if,in,local,%
    mod,not,od,or,%
    quit,rec,repeat,return,%
    then,true,until,while%
  },%
  sensitive,%
  morecomment=[l]\#,%
  morestring=[b]",%
  morestring=[b]',%
}[keywords,comments,strings]
\usepackage[T1]{fontenc}
\usepackage{xcolor}
\newtheorem{thm}{Theorem}[section]
\newtheorem{cor}[thm]{Corollary}
\newtheorem{lem}[thm]{Lemma}

\newcommand{\abs}[1]{\left\vert#1\right\vert}
\newcommand{\bb}[1]{\textcolor{blue}{{#1}}}

\newcommand{\J}{\mathrel{\mathscr J}} 
\newcommand{\R}{\mathrel{\mathscr R}} 
\newcommand{\eL}{\mathrel{\mathscr L}} 
\newcommand{\HH}{\mathrel{\mathscr H}}
\newcommand\MMM{\stackrel{\mathclap{\normalfont\mbox{\tiny $M$}}}{=}}
\newcommand\JJJ{\stackrel{\mathclap{\normalfont\mbox{\tiny $\mathcal{J}_5$}}}{=}}
\begin{document}
\title{Identities of the Jones Monoid $\mathcal{J}_5$}
\author{M. H. Shahzamanian}
\address{M. H. Shahzamanian\\ Centro de Matemática e Departamento de Matemática, Faculdade de Ciências,
Universidade do Porto, Rua do Campo Alegre, 687, 4169-007 Porto,
Portugal}
\email{m.h.shahzamanian@fc.up.pt}

\begin{abstract}
Jones monoids $\mathcal{J}_n$, for $1< n$, is a family of monoids relevant in knot theory.
The purpose of this paper is to characterize of the identities satisfied by the Jones monoid $\mathcal{J}_5$. 
\end{abstract}
\maketitle


\section{Introduction}
\label{sec:intro}

Let $1< n$.
The Kauffman monoid $\mathcal{K}_n$ is the monoid generated by
$$h_1,\dots,h_{n-1},c,$$ subject to the following relations:
\begin{equation}\label{Kn-relations}
\begin{split}
&h_{i}h_{j}=h_{j}h_{i},\   \forall 1\leq i,j\leq n-1,\ \ \text{if}\ \abs{i-j}\geq 2,\\
&h_{i}h_{j}h_{i}=h_{i},\   \forall 1\leq i,j\leq n-1,\ \ \text{if}\ \abs{i-j}=1,\\
&h_{i}^2=ch_{i}=h_{i}c,\   \forall 1\leq i\leq n-1.
\end{split}
\end{equation}
L. H. Kauffman in~\cite{Kau} invented these monoids, $\mathcal{K}_n$, as geometric objects. (The name was suggested in~\cite{Bor-Mir-Do-Pet}.)
Kauffman monoids play an important role in several parts of mathematics such as 
knot theory, low-dimensional topology, topological quantum field theory, quantum groups, etc.
Jones monoid $\mathcal{J}_n$ is the monoid generated by
$h_1,\dots,h_{n-1}$, subject to the first and second relations in~(\ref{Kn-relations}) and the relation $h_{i}^2=h_{i}$, for all $1\leq i\leq n-1$.
Jones monoid is a class of diagram monoids like Kauffman monoids. 
(The name was suggested in~\cite{Lau-Kwo-Fit} to honor the contribution of V. F. R. Jones to the theory). 

Chen et al. in~\cite{Che-Hu-Kit-Vol}, provide an algorithm for checking identities in $\mathcal{K}_3$. 
Kitov and Volkov in~\cite{Kit-Vol}, extend this algorithm to the Kauffman monoid $\mathcal{K}_4$ and also find a polynomial time algorithm for checking identities in the Jones monoid $\mathcal{J}_4$. They prove that the Kauffman monoids $\mathcal{K}_3$ and $\mathcal{K}_4$ satisfy
exactly the same identities. By delivering an identity, they show that $\mathcal{K}_4$ and $\mathcal{K}_5$ do not satisfy the same identities.
In the present paper, we follow this line of research and characterize the identities of the monoid $\mathcal{J}_5$. 
This characterization helps us to explain that for an identity when $\mathcal{J}_4$ satisfies it and $\mathcal{J}_5$ does not satisfy it.

The paper is organized as follows.  We begin by recalling background on monoids, identities and Jones monoids, so as to make the paper accessible to as broad an audience as possible.  
Also, in a separated section we investigate the Jones monoid $\mathcal{J}_5$ and give some lemmas which we need in the proof of the characterization.
We then present in the following section our characterization of identities in $\mathcal{J}_5$.


\section{Preliminaries}
\subsection{Monoids}
For standard notation and terminology relating to semigroups and monoids, we refer the reader to~\cite[Chapter 5]{Alm},\cite[Chapters 1-3]{Cli-Pre} and~\cite[Appendix A]{Rho-Ste}.
Let $M$ a finite monoid. Let $a,b\in M$. We say that $a\R b$ if $aM = bM$, $a\eL b$ if $Ma = Mb$ and $a\HH b$ if $a\R b$ and $a\eL b$. Also, we say that $a\J b$, if $MaM = MbM$.
The relations $\R,\eL$, $\HH$ and $\J$ are Green relations and all of them are equivalence relations first introduced by Green~\cite{Gre}.
An important property of finite monoids is the stability property that $J_m\cap Mm = L_m$ and $J_m\cap mM = R_m$, for every $m \in M$. 
A finite monoid is aperiodic if and only if its $\HH$-relation is trivial.

An element $e$ of $M$ is called idempotent if $e^2 = e$. The set of all idempotents of $M$ is denoted by $E(M)$. 
An idempotent $e$ of $M$ is the identity of the monoid $eMe$. The group of units $G_e$ of $eMe$ is called the maximal subgroup of $M$ at $e$. 

An element $m$ of $M$ is called (von Neumann) regular if there exists an element $n\in M$ such that $mnm=m$. Note that an element $m$ is regular if and only if $m\eL e$, for some $e\in E(M)$, if and only if $m\R f$, for some $f\in E(M)$. A $\J$-class $J$ is regular if all its elements are regular, if and only if $J$ has an idempotent, if and only if $J^2\cap J\neq\emptyset$. 


\subsection{Identities}
Let $X$ be a countably infinite set. 
We call $X$ an alphabet and each element $x\in X$ an letter.
Let $X^+$ be the set of all finite, non-empty words $x_1\cdots x_n$ with $x_1,\ldots,x_n\in X$.
The set $X^+$ forms a semigroup under concatenation which is called the free semigroup over $X$.
The monoid $X^*=(X^+)^1$ is called the free monoid over $X$.

Let $t = x_1\cdots x_n$ be a word of $X^+$ with $x_1,\ldots, x_n \in X$. 
The set $\{x_1,\ldots, x_n\}$ is called the content of $t$ and is denoted $c(t)$
while the number $n$ is referred to as the length of $t$ and is denoted $\abs{t}$. 
If $x \in c(t)$, we say that a letter $x$ occurs in a word $t$. 
We say that a word $s \in X^+$ occurs in $t$ if $t = t_1st_2$ for some $t_1, t_2 \in X^*$.
Let $u=u_1\cdots u_m$ be a word in $X^*$ with $u_1,\ldots,u_m\in X$.
We say that $u$ is a subword of the word $t$, 
if $t$ can be written $t = u'_0u_1u'_1\cdots u_mu'_m$ for some words $u'_0,u'_1,\ldots,u'_m \in X^*$. 
For a subset $Y$ of the set $X$, let $t_Y$ be the longest subword of $t$ with $c(t_Y)\subseteq Y$.

An identity is an expression $t_1=t_2$ with $t_1,t_2\in X^*$.
Let $M$ be a monoid.
We say that the identity $t_1=t_2$ holds in $M$ or $M$ satisfies the identity $t_1=t_2$ if $\phi(t_1) = \phi(t_2)$ for every homomorphism $\phi\colon X^*\rightarrow M$ and we denote it by $t_1 \MMM t_2$.


\subsection{Jones monoids}
Let $1< n$.
Jones monoid $\mathcal{J}_n$ is the monoid generated by 
$h_1,\dots,h_{n-1}$, subject to the following relations:
\begin{equation}\label{Jn-relations}
\begin{split}
&h_{i}h_{j}=h_{j}h_{i},    \forall i,j\in\{1,\dots,n-1\},\ \ \text{if}\ \abs{i-j}\geq 2,\\
&h_{i}h_{j}h_{i}=h_{i},    \forall i,j\in\{1,\dots,n-1\},\ \ \text{if}\ \abs{i-j}=1,\\
&h_{i}^2=h_{i},            \forall i\in\{1,\dots,n-1\}.
\end{split}
\end{equation}
Note that $\mathcal{J}_n$ is not a submonoid of $K_n$.

The Kauffman monoids $\mathcal{K}_n$ and Jones monoids $\mathcal{J}_n$ may be presented by geometric definitions with a series of diagram monoids (see \cite{Aui-Vol}). In the current paper, we only deal with the Jones monoid $\mathcal{J}_5$. Hence, we only mention a version of these geometric definitions which led to defining the Jones monoids.

Let $[n] \coloneq \{1, \dots , n\}$ and $[n]'\coloneq \{1', \ldots , n'\}$ be two disjoint copies of the set
of the first $n$ positive integers. 
Let $\mathcal{B}_n$ be the set of all partitions $\pi$ of the $2n$-element set $[n] \cup [n]'$ into 2-element blocks.
Such a pair can be represented by a wire diagram as shown in Figure~\ref{fig:Diag}. 
We draw a rectangular chip with $2n$ pins and represent the elements of $[n]$ by pins on the left hand side of the chip (left pins) while the
elements of $[n]'$ are represented by pins on the right hand side of the chip
(right pins). Usually we omit the numbers $1, 2, \ldots, n$. 
Now, for $\pi \in \mathcal{B}_n$, we represent each block of the partition $\pi$ is represented by
a line referred to as a wire. 
Thus, each wire connects two points; it is called
an $l$-wire if it connects two left points, an $r$-wire if it connects two right points,
and a $t$-wire if it connects a left point with a right point.
Thus, each wire connects two pins. The wire diagram in Figure~\ref{fig:Diag} corresponds to the pair
$\{\{1, 2\}, \{3, 5\}, \{4, 1'\}, \{2', 5'\}, \{3', 4'\} \}$.
\begin{figure}[ht]
\begin{center}\scriptsize
\begin{tikzpicture}
\filldraw[black] (0,0) circle (2pt) node[anchor=east]{$5~$};
\filldraw[black] (0,.5) circle (2pt) node[anchor=east]{$4~$};
\filldraw[black] (0,1) circle (2pt) node[anchor=east]{$3~$};
\filldraw[black] (0,1.5) circle (2pt) node[anchor=east]{$2~$};
\filldraw[black] (0,2) circle (2pt) node[anchor=east]{$1~$};
\filldraw[black] (1,0) circle (2pt) node[anchor=west]{$~5'$};
\filldraw[black] (1,.5) circle (2pt) node[anchor=west]{$~4'$};
\filldraw[black] (1,1) circle (2pt) node[anchor=west]{$~3'$};
\filldraw[black] (1,1.5) circle (2pt) node[anchor=west]{$~2'$};
\filldraw[black] (1,2) circle (2pt) node[anchor=west]{$~1'$};
\draw (0,2) edge [-, bend left=20] (0,1.5);
\draw (0,.5) edge [-, bend right=0] (1,2);
\draw (0,1) edge [-, bend left=20] (0,0);
\draw (1,1) edge [-, bend right=20] (1,.5);
\draw (1,1.5) edge [-, bend right=20] (1,0);
\draw [dashed] (-.11,-.11) rectangle (1.11,2.11);
\begin{scope}
\clip (0.25,2.115) rectangle (0.75,1.9);
\draw[thick] (0.5,2.1) circle(0.1);
\end{scope}
\end{tikzpicture}
\end{center}
    \caption{Wire diagram for an element of $\mathcal{B}_5$}
    \label{fig:Diag}
  \end{figure}
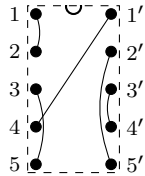
  
The multiply of two wire diagrams in $\mathcal{B}_n$, we shortcut the right pins of the first chip with the corresponding
left pins of the second chip.
Thus, we obtain a new chip whose left pins are the left pins of the first chip, right pins are the right pins of the second chip whose wires are sequences of consecutive wires of the factors, see Figure~\ref{fig:Mul} (for more detail refer to \cite{Aui-Vol}). 
\begin{figure}[ht]
\begin{center}\scriptsize
\begin{tikzpicture}
\filldraw[black] (0,0) circle (2pt) node[anchor=east]{};
\filldraw[black] (0,.5) circle (2pt) node[anchor=east]{};
\filldraw[black] (0,1) circle (2pt) node[anchor=east]{};
\filldraw[black] (0,1.5) circle (2pt) node[anchor=east]{};
\filldraw[black] (0,2) circle (2pt) node[anchor=east]{};
\filldraw[black] (1,0) circle (2pt) node[anchor=west]{};
\filldraw[black] (1,.5) circle (2pt) node[anchor=west]{};
\filldraw[black] (1,1) circle (2pt) node[anchor=west]{};
\filldraw[black] (1,1.5) circle (2pt) node[anchor=west]{};
\filldraw[black] (1,2) circle (2pt) node[anchor=west]{};
\draw (0,2) edge [-, bend left=20] (0,1.5);
\draw (0,.5) edge [-, bend right=0] (1,2);
\draw (0,1) edge [-, bend left=20] (0,0);
\draw (1,1) edge [-, bend right=20] (1,.5);
\draw (1,1.5) edge [-, bend right=20] (1,0);
\draw [dashed] (-.12,-.12) rectangle (1.12,2.12);
\node at (1.5,1.1) {$\bf{\times}$};
\begin{scope}
\clip (0.25,2.115) rectangle (0.75,1.9);
\draw[thick] (0.5,2.1) circle(0.1);
\end{scope}

\filldraw[black] (2,0) circle (2pt) node[anchor=east]{};
\filldraw[black] (2,.5) circle (2pt) node[anchor=east]{};
\filldraw[black] (2,1) circle (2pt) node[anchor=east]{};
\filldraw[black] (2,1.5) circle (2pt) node[anchor=east]{};
\filldraw[black] (2,2) circle (2pt) node[anchor=east]{};
\filldraw[black] (3,0) circle (2pt) node[anchor=west]{};
\filldraw[black] (3,.5) circle (2pt) node[anchor=west]{};
\filldraw[black] (3,1) circle (2pt) node[anchor=west]{};
\filldraw[black] (3,1.5) circle (2pt) node[anchor=west]{};
\filldraw[black] (3,2) circle (2pt) node[anchor=west]{};
\draw (2,2) edge [-, bend left=0] (3,1.5);
\draw (2,1.5) edge [-, bend right=0] (3,2);
\draw (2,1) edge [-, bend left=0] (3,1);
\draw (2,.5) edge [-, bend right=0] (3,0);
\draw (2,0) edge [-, bend right=0] (3,.5);
\draw [dashed] (1.88,-.12) rectangle (3.12,2.12);
\node at (3.5,1.1) {$\bf{=}$};
\begin{scope}
\clip (2.25,2.115) rectangle (3,1.9);
\draw[thick] (2.5,2.1) circle(0.1);
\end{scope}

\filldraw[black] (4,0) circle (2pt) node[anchor=east]{};
\filldraw[black] (4,.5) circle (2pt) node[anchor=east]{};
\filldraw[black] (4,1) circle (2pt) node[anchor=east]{};
\filldraw[black] (4,1.5) circle (2pt) node[anchor=east]{};
\filldraw[black] (4,2) circle (2pt) node[anchor=east]{};
\filldraw[black] (5,0) circle (2pt) node[anchor=west]{};
\filldraw[black] (5,.5) circle (2pt) node[anchor=west]{};
\filldraw[black] (5,1) circle (2pt) node[anchor=west]{};
\filldraw[black] (5,1.5) circle (2pt) node[anchor=west]{};
\filldraw[black] (5,2) circle (2pt) node[anchor=west]{};
\filldraw[black] (6,0) circle (2pt) node[anchor=west]{};
\filldraw[black] (6,.5) circle (2pt) node[anchor=west]{};
\filldraw[black] (6,1) circle (2pt) node[anchor=west]{};
\filldraw[black] (6,1.5) circle (2pt) node[anchor=west]{};
\filldraw[black] (6,2) circle (2pt) node[anchor=west]{};
\draw (4,2) edge [-, bend left=20] (4,1.5);
\draw (4,.5) edge [-, bend right=0] (5,2);
\draw (4,1) edge [-, bend left=20] (4,0);
\draw (5,1) edge [-, bend right=20] (5,.5);
\draw (5,1.5) edge [-, bend right=20] (5,0);
\draw (5,2) edge [-, bend left=0] (6,1.5);
\draw (5,1.5) edge [-, bend right=0] (6,2);
\draw (5,1) edge [-, bend left=0] (6,1);
\draw (5,.5) edge [-, bend right=0] (6,0);
\draw (5,0) edge [-, bend right=0] (6,.5);
\draw [dashed] (3.88,-0.12) rectangle (6.12,2.12);
\node at (6.5,1.1) {$\bf{=}$};
\begin{scope}
\clip (4.25,2.115) rectangle (4.75,1.9);
\draw[thick] (4.5,2.1) circle(0.1);
\end{scope}
\begin{scope}
\clip (5.25,2.115) rectangle (5.75,1.9);
\draw[thick] (5.5,2.1) circle(0.1);
\end{scope}

\filldraw[black] (7,0) circle (2pt) node[anchor=east]{};
\filldraw[black] (7,.5) circle (2pt) node[anchor=east]{};
\filldraw[black] (7,1) circle (2pt) node[anchor=east]{};
\filldraw[black] (7,1.5) circle (2pt) node[anchor=east]{};
\filldraw[black] (7,2) circle (2pt) node[anchor=east]{};
\filldraw[black] (8,0) circle (2pt) node[anchor=west]{};
\filldraw[black] (8,.5) circle (2pt) node[anchor=west]{};
\filldraw[black] (8,1) circle (2pt) node[anchor=west]{};
\filldraw[black] (8,1.5) circle (2pt) node[anchor=west]{};
\filldraw[black] (8,2) circle (2pt) node[anchor=west]{};
\draw (7,2) edge [-, bend left=20] (7,1.5);
\draw (7,1) edge [-, bend left=20] (7,0);
\draw (7,.5) edge [-, bend left=0] (8,1.5);
\draw (8,2) edge [-, bend right=20] (8,.5);
\draw (8,1) edge [-, bend right=20] (8,0);
\draw [dashed] (6.88,-0.12) rectangle (8.12,2.12);
\begin{scope}
\clip (7.25,2.115) rectangle (7.75,1.9);
\draw[ thick] (7.5,2.1) circle(0.1);
\end{scope}
\end{tikzpicture}
\end{center}
    \caption{Multiplication of wire diagrams}
    \label{fig:Mul}
  \end{figure}
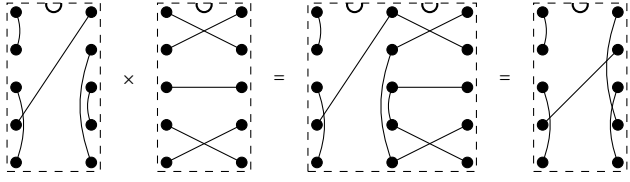

It is easy to see that the above defined multiplication in $\mathcal{B}_n$ is associative and
that the chip corresponds to the pair
$\{\{1, 1'\}, \{2, 2'\},\ldots  \{n, n'\}\}$ is the identity element with respect to the multiplication. The monoid $\mathcal{B}_n$ is known as the Brauer monoid~\cite{Bra}.
The Jones monoid $\mathcal{J}_n$ is the submonoid of $\mathcal{B}_n$ consisting of all elements of $\mathcal{B}_n$ that have a representation as a chip whose wires do not cross. The element $h_i$ in $\mathcal{J}_n$, for $1\leq i\leq n-1$, is the chip corresponds to the pair
$$\{\{i, i+1\}, \{i', (i+1)'\},\{j,j'\}\mid \text{for all}\ j\neq i, i+1)\},$$ see Figure~\ref{fig:hi-Jn}. These chips satisfy the relations~(\ref{Jn-relations}). Note that the cardinality of $\mathcal{J}_n$ is equal to $\frac{1}{n+1}\binom{2n}{n}$.
\begin{figure}[ht]
\begin{center}\scriptsize
\begin{tikzpicture}
\filldraw[black] (0,0) circle (2pt) node[anchor=east]{$n~$};
\node at (0,0.5) {$\bf{\vdots}$};
\filldraw[black] (0,1) circle (2pt) node[anchor=east]{$i+1~$};
\filldraw[black] (0,1.5) circle (2pt) node[anchor=east]{$i~$};
\node at (0,2) {$\vdots$};
\filldraw[black] (0,2.5) circle (2pt) node[anchor=east]{$2~$};
\filldraw[black] (0,3) circle (2pt) node[anchor=east]{$1~$};
\filldraw[black] (1,0) circle (2pt) node[anchor=west]{$~n'$};
\node at (1,.5) {$\vdots$};
\filldraw[black] (1,1) circle (2pt) node[anchor=west]{$~(i+1)'$};
\filldraw[black] (1,1.5) circle (2pt) node[anchor=west]{$~i'$};
\node at (1,2) {$\vdots$};
\filldraw[black] (1,2.5) circle (2pt) node[anchor=west]{$~2'$};
\filldraw[black] (1,3) circle (2pt) node[anchor=west]{$~1'$};
\draw (0,0) edge [-, bend left=0] (1,0);
\draw (0,2.5) edge [-, bend left=0] (1,2.5);
\draw (0,3) edge [-, bend left=0] (1,3);
\draw (0,1) edge [-, bend right=20] (0,1.5);
\draw (1,1) edge [-, bend left=20] (1,1.5);
\draw [dashed] (-.11,-.11) rectangle (1.11,3.11);
\begin{scope}
\clip (0.25,3.12) rectangle (0.75,2.8);
\draw[thick] (0.5,3.14) circle(0.1);
\end{scope}
\end{tikzpicture}
\end{center}
    \caption{Wire diagram for an element of $\mathcal{B}_5$}
    \label{fig:hi-Jn}
  \end{figure}
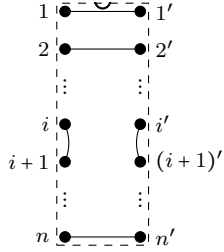
  

\section{The Jones monoid $\mathcal{J}_5$}

The Jones monoid $\mathcal{J}_5$ is aperiodic and regular and has three $\J$-classes named $A_1=\{1\}$, $A_2$ and $A_3$. 
The elements of the $\J$-class $A_2$ have one $l$-wire and one $r$-wire, and the elements of the $\J$-class $A_3$ have two $l$-wire and two $r$-wire.
The elements $h_1$, $h_2$, $h_3$ and $h_4$ are in $A_2$, and the elements $h_1h_3$, $h_1h_4$ and $h_2h_4$ are in $A_3$. Since $h_{i}h_{j}h_{i}=h_{i}$ for all $1\leq i,j\leq 4$ with $\abs{i-j}=1$ and $h_{i}^2=h_{i}$ for all $1\leq i\leq 4$, for every $a\in A_2\setminus\{h_1,h_2,h_3,h_4\}$, there exist integers $1\leq i,j\leq 4$ such that 
\begin{equation*}
a  = \begin{cases}
  h_{i}h_{i+1}\cdots h_{j-1}h_{j}& \text{if}\ i<j;\\
  h_{i}h_{i-1}\cdots h_{j+1}h_{j}& \text{if}\ i>j.
\end{cases}
\end{equation*}
The $\J$-class $A_2$ has four $\R$-classes and four $\eL$-classes (see Figure~\ref{fig:A2}). In Figure~\ref{fig:A2}, the rows corresponding to the $\R$-classes and the columns to the $\eL$-classes contained in $A_2$. 
\begin{figure}[ht]
\begin{center}\scriptsize
\begin{tikzpicture}
\filldraw[black] (0,0) circle (2pt) node[anchor=east]{};
\filldraw[black] (0,.5) circle (2pt) node[anchor=east]{};
\filldraw[black] (0,1) circle (2pt) node[anchor=east]{};
\filldraw[black] (0,1.5) circle (2pt) node[anchor=east]{};
\filldraw[black] (0,2) circle (2pt) node[anchor=east]{};
\filldraw[black] (1,0) circle (2pt) node[anchor=west]{};
\filldraw[black] (1,.5) circle (2pt) node[anchor=west]{};
\filldraw[black] (1,1) circle (2pt) node[anchor=west]{};
\filldraw[black] (1,1.5) circle (2pt) node[anchor=west]{};
\filldraw[black] (1,2) circle (2pt) node[anchor=west]{};
\draw (0,0) edge [-, bend left=0] (1,0);
\draw (0,.5) edge [-, bend right=0] (1,.5);
\draw (0,1) edge [-, bend left=0] (1,1);
\draw (0,1.5) edge [-, bend right=20] (0,2);
\draw (1,1.5) edge [-, bend left=20] (1,2);
\draw [dashed] (-.12,-.12) rectangle (1.12,2.12);
\node at (.5,-.3) {$h_1$};
\begin{scope}
\clip (0.25,2.115) rectangle (0.75,1.9);
\draw[thick] (0.5,2.14) circle(0.1);
\end{scope}

\filldraw[black] (2,0) circle (2pt) node[anchor=east]{};
\filldraw[black] (2,.5) circle (2pt) node[anchor=east]{};
\filldraw[black] (2,1) circle (2pt) node[anchor=east]{};
\filldraw[black] (2,1.5) circle (2pt) node[anchor=east]{};
\filldraw[black] (2,2) circle (2pt) node[anchor=east]{};
\filldraw[black] (3,0) circle (2pt) node[anchor=west]{};
\filldraw[black] (3,.5) circle (2pt) node[anchor=west]{};
\filldraw[black] (3,1) circle (2pt) node[anchor=west]{};
\filldraw[black] (3,1.5) circle (2pt) node[anchor=west]{};
\filldraw[black] (3,2) circle (2pt) node[anchor=west]{};
\draw (2,0) edge [-, bend left=0] (3,0);
\draw (2,.5) edge [-, bend right=0] (3,.5);
\draw (2,1) edge [-, bend left=0] (3,2);
\draw (2,1.5) edge [-, bend right=20] (2,2);
\draw (3,1.5) edge [-, bend right=20] (3,1);
\draw [dashed] (1.88,-.12) rectangle (3.12,2.12);
\node at (2.5,-.3) {$h_1h_2$};
\begin{scope}
\clip (2.25,2.115) rectangle (2.75,1.9);
\draw[thick] (2.5,2.14) circle(0.1);
\end{scope}

\filldraw[black] (4,0) circle (2pt) node[anchor=east]{};
\filldraw[black] (4,.5) circle (2pt) node[anchor=east]{};
\filldraw[black] (4,1) circle (2pt) node[anchor=east]{};
\filldraw[black] (4,1.5) circle (2pt) node[anchor=east]{};
\filldraw[black] (4,2) circle (2pt) node[anchor=east]{};
\filldraw[black] (5,0) circle (2pt) node[anchor=west]{};
\filldraw[black] (5,.5) circle (2pt) node[anchor=west]{};
\filldraw[black] (5,1) circle (2pt) node[anchor=west]{};
\filldraw[black] (5,1.5) circle (2pt) node[anchor=west]{};
\filldraw[black] (5,2) circle (2pt) node[anchor=west]{};
\draw (4,0) edge [-, bend left=0] (5,0);
\draw (4,.5) edge [-, bend right=0] (5,1.5);
\draw (4,1) edge [-, bend left=0] (5,2);
\draw (4,1.5) edge [-, bend right=20] (4,2);
\draw (5,1) edge [-, bend right=20] (5,.5);
\draw [dashed] (3.88,-.12) rectangle (5.12,2.12);
\node at (4.5,-.3) {$h_1h_2h_3$};
\begin{scope}
\clip (4.25,2.115) rectangle (4.75,1.9);
\draw[thick] (4.5,2.14) circle(0.1);
\end{scope}

\filldraw[black] (6,0) circle (2pt) node[anchor=east]{};
\filldraw[black] (6,.5) circle (2pt) node[anchor=east]{};
\filldraw[black] (6,1) circle (2pt) node[anchor=east]{};
\filldraw[black] (6,1.5) circle (2pt) node[anchor=east]{};
\filldraw[black] (6,2) circle (2pt) node[anchor=east]{};
\filldraw[black] (7,0) circle (2pt) node[anchor=west]{};
\filldraw[black] (7,.5) circle (2pt) node[anchor=west]{};
\filldraw[black] (7,1) circle (2pt) node[anchor=west]{};
\filldraw[black] (7,1.5) circle (2pt) node[anchor=west]{};
\filldraw[black] (7,2) circle (2pt) node[anchor=west]{};
\draw (6,0) edge [-, bend left=0] (7,1);
\draw (6,.5) edge [-, bend right=0] (7,1.5);
\draw (6,1) edge [-, bend left=0] (7,2);
\draw (6,1.5) edge [-, bend right=20] (6,2);
\draw (7,0) edge [-, bend left=20] (7,.5);
\draw [dashed] (5.88,-.12) rectangle (7.12,2.12);
\node at (6.5,-.3) {$h_1h_2h_3h_4$};
\begin{scope}
\clip (6.25,2.115) rectangle (6.75,1.9);
\draw[thick] (6.5,2.14) circle(0.1);
\end{scope}

\filldraw[black] (0,-3) circle (2pt) node[anchor=east]{};
\filldraw[black] (0,-2.5) circle (2pt) node[anchor=east]{};
\filldraw[black] (0,-2) circle (2pt) node[anchor=east]{};
\filldraw[black] (0,-1.5) circle (2pt) node[anchor=east]{};
\filldraw[black] (0,-1) circle (2pt) node[anchor=east]{};
\filldraw[black] (1,-3) circle (2pt) node[anchor=west]{};
\filldraw[black] (1,-2.5) circle (2pt) node[anchor=west]{};
\filldraw[black] (1,-2) circle (2pt) node[anchor=west]{};
\filldraw[black] (1,-1.5) circle (2pt) node[anchor=west]{};
\filldraw[black] (1,-1) circle (2pt) node[anchor=west]{};
\draw (0,-3) edge [-, bend left=0] (1,-3);
\draw (0,-2.5) edge [-, bend right=0] (1,-2.5);
\draw (0,-2) edge [-, bend right=20] (0,-1.5);
\draw (0,-1) edge [-, bend right=0] (1,-2);
\draw (1,-1.5) edge [-, bend left=20] (1,-1);
\draw [dashed] (-.12,-3.12) rectangle (1.12,-.88);
\node at (.5,-3.3) {$h_2h_1$};
\begin{scope}
\clip (0.25,-.885) rectangle (0.75,-1.1);
\draw[thick] (0.5,-.86) circle(0.1);
\end{scope}

\filldraw[black] (2,-3) circle (2pt) node[anchor=east]{};
\filldraw[black] (2,-2.5) circle (2pt) node[anchor=east]{};
\filldraw[black] (2,-2) circle (2pt) node[anchor=east]{};
\filldraw[black] (2,-1.5) circle (2pt) node[anchor=east]{};
\filldraw[black] (2,-1) circle (2pt) node[anchor=east]{};
\filldraw[black] (3,-3) circle (2pt) node[anchor=west]{};
\filldraw[black] (3,-2.5) circle (2pt) node[anchor=west]{};
\filldraw[black] (3,-2) circle (2pt) node[anchor=west]{};
\filldraw[black] (3,-1.5) circle (2pt) node[anchor=west]{};
\filldraw[black] (3,-1) circle (2pt) node[anchor=west]{};
\draw (2,-3) edge [-, bend left=0] (3,-3);
\draw (2,-2.5) edge [-, bend right=0] (3,-2.5);
\draw (2,-1) edge [-, bend left=0] (3,-1);
\draw (2,-1.5) edge [-, bend left=20] (2,-2);
\draw (3,-1.5) edge [-, bend right=20] (3,-2);
\draw [dashed] (1.88,-3.12) rectangle (3.12,-.88);
\node at (2.5,-3.3) {$h_2$};
\begin{scope}
\clip (2.25,-.885) rectangle (2.75,-1.1);
\draw[thick] (2.5,-.86) circle(0.1);
\end{scope}

\filldraw[black] (4,-3) circle (2pt) node[anchor=east]{};
\filldraw[black] (4,-2.5) circle (2pt) node[anchor=east]{};
\filldraw[black] (4,-2) circle (2pt) node[anchor=east]{};
\filldraw[black] (4,-1.5) circle (2pt) node[anchor=east]{};
\filldraw[black] (4,-1) circle (2pt) node[anchor=east]{};
\filldraw[black] (5,-3) circle (2pt) node[anchor=west]{};
\filldraw[black] (5,-2.5) circle (2pt) node[anchor=west]{};
\filldraw[black] (5,-2) circle (2pt) node[anchor=west]{};
\filldraw[black] (5,-1.5) circle (2pt) node[anchor=west]{};
\filldraw[black] (5,-1) circle (2pt) node[anchor=west]{};
\draw (4,-3) edge [-, bend left=0] (5,-3);
\draw (4,-2.5) edge [-, bend right=0] (5,-1.5);
\draw (4,-1) edge [-, bend left=0] (5,-1);
\draw (4,-1.5) edge [-, bend left=20] (4,-2);
\draw (5,-2) edge [-, bend right=20] (5,-2.5);
\draw [dashed] (3.88,-3.12) rectangle (5.12,-.88);
\node at (4.5,-3.3) {$h_2h_3$};
\begin{scope}
\clip (4.25,-.885) rectangle (4.75,-1.1);
\draw[thick] (4.5,-.86) circle(0.1);
\end{scope}

\filldraw[black] (6,-3) circle (2pt) node[anchor=east]{};
\filldraw[black] (6,-2.5) circle (2pt) node[anchor=east]{};
\filldraw[black] (6,-2) circle (2pt) node[anchor=east]{};
\filldraw[black] (6,-1.5) circle (2pt) node[anchor=east]{};
\filldraw[black] (6,-1) circle (2pt) node[anchor=east]{};
\filldraw[black] (7,-3) circle (2pt) node[anchor=west]{};
\filldraw[black] (7,-2.5) circle (2pt) node[anchor=west]{};
\filldraw[black] (7,-2) circle (2pt) node[anchor=west]{};
\filldraw[black] (7,-1.5) circle (2pt) node[anchor=west]{};
\filldraw[black] (7,-1) circle (2pt) node[anchor=west]{};
\draw (6,-3) edge [-, bend left=0] (7,-2);
\draw (6,-2.5) edge [-, bend right=0] (7,-1.5);
\draw (6,-1) edge [-, bend left=0] (7,-1);
\draw (6,-1.5) edge [-, bend left=20] (6,-2);
\draw (7,-3) edge [-, bend left=20] (7,-2.5);
\draw [dashed] (5.88,-3.12) rectangle (7.12,-.88);
\node at (6.5,-3.3) {$h_2h_3h_4$};
\begin{scope}
\clip (6.25,-.885) rectangle (6.75,-1.1);
\draw[thick] (6.5,-.86) circle(0.1);
\end{scope}

\filldraw[black] (0,-6) circle (2pt) node[anchor=east]{};
\filldraw[black] (0,-5.5) circle (2pt) node[anchor=east]{};
\filldraw[black] (0,-5) circle (2pt) node[anchor=east]{};
\filldraw[black] (0,-4.5) circle (2pt) node[anchor=east]{};
\filldraw[black] (0,-4) circle (2pt) node[anchor=east]{};
\filldraw[black] (1,-6) circle (2pt) node[anchor=west]{};
\filldraw[black] (1,-5.5) circle (2pt) node[anchor=west]{};
\filldraw[black] (1,-5) circle (2pt) node[anchor=west]{};
\filldraw[black] (1,-4.5) circle (2pt) node[anchor=west]{};
\filldraw[black] (1,-4) circle (2pt) node[anchor=west]{};
\draw (0,-6) edge [-, bend left=0] (1,-6);
\draw (0,-4.5) edge [-, bend right=0] (1,-5.5);
\draw (0,-4) edge [-, bend left=0] (1,-5);
\draw (0,-5.5) edge [-, bend right=20] (0,-5);
\draw (1,-4.5) edge [-, bend left=20] (1,-4);
\draw [dashed] (-.12,-6.12) rectangle (1.12,-3.88);
\node at (.5,-6.3) {$h_3h_2h_1$};
\begin{scope}
\clip (0.25,-3.885) rectangle (0.75,-4.1);
\draw[thick] (0.5,-3.86) circle(0.1);
\end{scope}

\filldraw[black] (2,-6) circle (2pt) node[anchor=east]{};
\filldraw[black] (2,-5.5) circle (2pt) node[anchor=east]{};
\filldraw[black] (2,-5) circle (2pt) node[anchor=east]{};
\filldraw[black] (2,-4.5) circle (2pt) node[anchor=east]{};
\filldraw[black] (2,-4) circle (2pt) node[anchor=east]{};
\filldraw[black] (3,-6) circle (2pt) node[anchor=west]{};
\filldraw[black] (3,-5.5) circle (2pt) node[anchor=west]{};
\filldraw[black] (3,-5) circle (2pt) node[anchor=west]{};
\filldraw[black] (3,-4.5) circle (2pt) node[anchor=west]{};
\filldraw[black] (3,-4) circle (2pt) node[anchor=west]{};
\draw (2,-6) edge [-, bend left=0] (3,-6);
\draw (2,-4.5) edge [-, bend right=0] (3,-5.5);
\draw (2,-4) edge [-, bend left=0] (3,-4);
\draw (2,-5.5) edge [-, bend right=20] (2,-5);
\draw (3,-4.5) edge [-, bend right=20] (3,-5);
\draw [dashed] (1.88,-6.12) rectangle (3.12,-3.88);
\node at (2.5,-6.3) {$h_3h_2$};
\begin{scope}
\clip (2.25,-3.885) rectangle (2.75,-4.1);
\draw[thick] (2.5,-3.86) circle(0.1);
\end{scope}

\filldraw[black] (4,-6) circle (2pt) node[anchor=east]{};
\filldraw[black] (4,-5.5) circle (2pt) node[anchor=east]{};
\filldraw[black] (4,-5) circle (2pt) node[anchor=east]{};
\filldraw[black] (4,-4.5) circle (2pt) node[anchor=east]{};
\filldraw[black] (4,-4) circle (2pt) node[anchor=east]{};
\filldraw[black] (5,-6) circle (2pt) node[anchor=west]{};
\filldraw[black] (5,-5.5) circle (2pt) node[anchor=west]{};
\filldraw[black] (5,-5) circle (2pt) node[anchor=west]{};
\filldraw[black] (5,-4.5) circle (2pt) node[anchor=west]{};
\filldraw[black] (5,-4) circle (2pt) node[anchor=west]{};
\draw (4,-6) edge [-, bend left=0] (5,-6);
\draw (4,-4.5) edge [-, bend right=0] (5,-4.5);
\draw (4,-4) edge [-, bend left=0] (5,-4);
\draw (4,-5.5) edge [-, bend right=20] (4,-5);
\draw (5,-5) edge [-, bend right=20] (5,-5.5);
\draw [dashed] (3.88,-6.12) rectangle (5.12,-3.88);
\node at (4.5,-6.3) {$h_3$};
\begin{scope}
\clip (4.25,-3.885) rectangle (4.75,-4.1);
\draw[thick] (4.5,-3.86) circle(0.1);
\end{scope}

\filldraw[black] (6,-6) circle (2pt) node[anchor=east]{};
\filldraw[black] (6,-5.5) circle (2pt) node[anchor=east]{};
\filldraw[black] (6,-5) circle (2pt) node[anchor=east]{};
\filldraw[black] (6,-4.5) circle (2pt) node[anchor=east]{};
\filldraw[black] (6,-4) circle (2pt) node[anchor=east]{};
\filldraw[black] (7,-6) circle (2pt) node[anchor=west]{};
\filldraw[black] (7,-5.5) circle (2pt) node[anchor=west]{};
\filldraw[black] (7,-5) circle (2pt) node[anchor=west]{};
\filldraw[black] (7,-4.5) circle (2pt) node[anchor=west]{};
\filldraw[black] (7,-4) circle (2pt) node[anchor=west]{};
\draw (6,-6) edge [-, bend left=0] (7,-5);
\draw (6,-4.5) edge [-, bend right=0] (7,-4.5);
\draw (6,-4) edge [-, bend left=0] (7,-4);
\draw (6,-5.5) edge [-, bend right=20] (6,-5);
\draw (7,-6) edge [-, bend left=20] (7,-5.5);
\draw [dashed] (5.88,-6.12) rectangle (7.12,-3.88);
\node at (6.5,-6.3) {$h_3h_4$};
\begin{scope}
\clip (6.25,-3.885) rectangle (6.75,-4.1);
\draw[thick] (6.5,-3.86) circle(0.1);
\end{scope}

\filldraw[black] (0,-9) circle (2pt) node[anchor=east]{};
\filldraw[black] (0,-8.5) circle (2pt) node[anchor=east]{};
\filldraw[black] (0,-8) circle (2pt) node[anchor=east]{};
\filldraw[black] (0,-7.5) circle (2pt) node[anchor=east]{};
\filldraw[black] (0,-7) circle (2pt) node[anchor=east]{};
\filldraw[black] (1,-9) circle (2pt) node[anchor=west]{};
\filldraw[black] (1,-8.5) circle (2pt) node[anchor=west]{};
\filldraw[black] (1,-8) circle (2pt) node[anchor=west]{};
\filldraw[black] (1,-7.5) circle (2pt) node[anchor=west]{};
\filldraw[black] (1,-7) circle (2pt) node[anchor=west]{};
\draw (0,-8) edge [-, bend left=0] (1,-9);
\draw (0,-7.5) edge [-, bend right=0] (1,-8.5);
\draw (0,-7) edge [-, bend left=0] (1,-8);
\draw (0,-9) edge [-, bend right=20] (0,-8.5);
\draw (1,-7.5) edge [-, bend left=20] (1,-7);
\draw [dashed] (-.12,-9.12) rectangle (1.12,-6.88);
\node at (.5,-9.3) {$h_4h_3h_2h_1$};
\begin{scope}
\clip (0.25,-6.885) rectangle (0.75,-7.1);
\draw[thick] (0.5,-6.86) circle(0.1);
\end{scope}

\filldraw[black] (2,-9) circle (2pt) node[anchor=east]{};
\filldraw[black] (2,-8.5) circle (2pt) node[anchor=east]{};
\filldraw[black] (2,-8) circle (2pt) node[anchor=east]{};
\filldraw[black] (2,-7.5) circle (2pt) node[anchor=east]{};
\filldraw[black] (2,-7) circle (2pt) node[anchor=east]{};
\filldraw[black] (3,-9) circle (2pt) node[anchor=west]{};
\filldraw[black] (3,-8.5) circle (2pt) node[anchor=west]{};
\filldraw[black] (3,-8) circle (2pt) node[anchor=west]{};
\filldraw[black] (3,-7.5) circle (2pt) node[anchor=west]{};
\filldraw[black] (3,-7) circle (2pt) node[anchor=west]{};
\draw (2,-8) edge [-, bend left=0] (3,-9);
\draw (2,-7.5) edge [-, bend right=0] (3,-8.5);
\draw (2,-7) edge [-, bend left=0] (3,-7);
\draw (2,-9) edge [-, bend right=20] (2,-8.5);
\draw (3,-7.5) edge [-, bend right=20] (3,-8);
\draw [dashed] (1.88,-9.12) rectangle (3.12,-6.88);
\node at (2.5,-9.3) {$h_4h_3h_2$};
\begin{scope}
\clip (2.25,-6.885) rectangle (2.75,-7.1);
\draw[thick] (2.5,-6.86) circle(0.1);
\end{scope}

\filldraw[black] (4,-9) circle (2pt) node[anchor=east]{};
\filldraw[black] (4,-8.5) circle (2pt) node[anchor=east]{};
\filldraw[black] (4,-8) circle (2pt) node[anchor=east]{};
\filldraw[black] (4,-7.5) circle (2pt) node[anchor=east]{};
\filldraw[black] (4,-7) circle (2pt) node[anchor=east]{};
\filldraw[black] (5,-9) circle (2pt) node[anchor=west]{};
\filldraw[black] (5,-8.5) circle (2pt) node[anchor=west]{};
\filldraw[black] (5,-8) circle (2pt) node[anchor=west]{};
\filldraw[black] (5,-7.5) circle (2pt) node[anchor=west]{};
\filldraw[black] (5,-7) circle (2pt) node[anchor=west]{};
\draw (4,-8) edge [-, bend left=0] (5,-9);
\draw (4,-7.5) edge [-, bend right=0] (5,-7.5);
\draw (4,-7) edge [-, bend left=0] (5,-7);
\draw (4,-9) edge [-, bend right=20] (4,-8.5);
\draw (5,-8) edge [-, bend right=20] (5,-8.5);
\draw [dashed] (3.88,-9.12) rectangle (5.12,-6.88);
\node at (4.5,-9.3) {$h_4h_3$};
\begin{scope}
\clip (4.25,-6.885) rectangle (4.75,-7.1);
\draw[thick] (4.5,-6.86) circle(0.1);
\end{scope}

\filldraw[black] (6,-9) circle (2pt) node[anchor=east]{};
\filldraw[black] (6,-8.5) circle (2pt) node[anchor=east]{};
\filldraw[black] (6,-8) circle (2pt) node[anchor=east]{};
\filldraw[black] (6,-7.5) circle (2pt) node[anchor=east]{};
\filldraw[black] (6,-7) circle (2pt) node[anchor=east]{};
\filldraw[black] (7,-9) circle (2pt) node[anchor=west]{};
\filldraw[black] (7,-8.5) circle (2pt) node[anchor=west]{};
\filldraw[black] (7,-8) circle (2pt) node[anchor=west]{};
\filldraw[black] (7,-7.5) circle (2pt) node[anchor=west]{};
\filldraw[black] (7,-7) circle (2pt) node[anchor=west]{};
\draw (6,-8) edge [-, bend left=0] (7,-8);
\draw (6,-7.5) edge [-, bend right=0] (7,-7.5);
\draw (6,-7) edge [-, bend left=0] (7,-7);
\draw (6,-9) edge [-, bend right=20] (6,-8.5);
\draw (7,-9) edge [-, bend left=20] (7,-8.5);
\draw [dashed] (5.88,-9.12) rectangle (7.12,-6.88);
\node at (6.5,-9.3) {$h_4$};
\begin{scope}
\clip (6.25,-6.885) rectangle (6.75,-7.1);
\draw[thick] (6.5,-6.86) circle(0.1);
\end{scope}
\end{tikzpicture}
\end{center}
    \caption{The elements of the $\J$-class $A_2$}
    \label{fig:A2}
\end{figure}
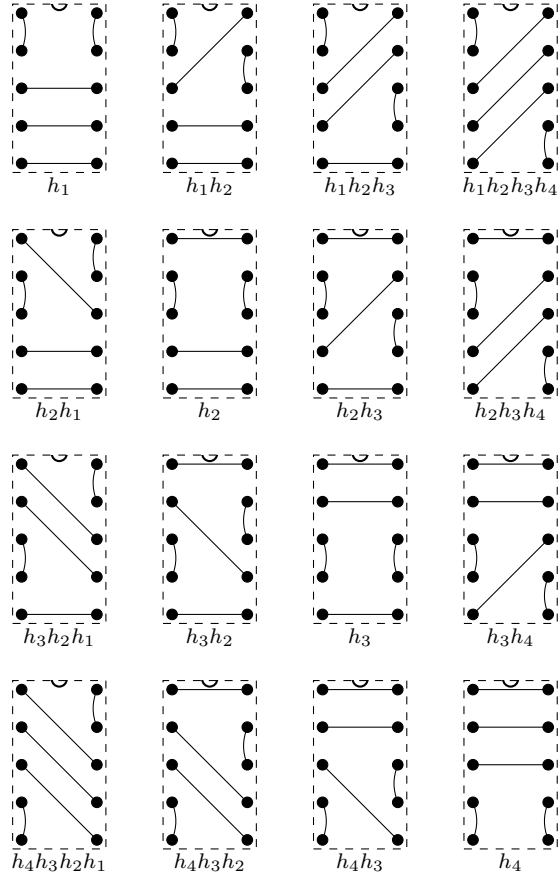

If $a\in A_3$, then there exist elements $a_1, a_2\in J_5$ and $b\in\{h_1h_3,h_1h_4,h_2h_4\}$ such that $a=a_1ba_2$.
The $\J$-class $A_2$ has five $\R$-classes and five $\eL$-classes (see Figure~\ref{fig:A3}). In Figure~\ref{fig:A3}, the rows corresponding to the $\R$-classes and the columns to the $\eL$-classes contained in $A_3$. 
\begin{figure}[ht]
\begin{center}\scriptsize
\begin{tikzpicture}
\filldraw[black] (0,0) circle (2pt) node[anchor=east]{};
\filldraw[black] (0,.5) circle (2pt) node[anchor=east]{};
\filldraw[black] (0,1) circle (2pt) node[anchor=east]{};
\filldraw[black] (0,1.5) circle (2pt) node[anchor=east]{};
\filldraw[black] (0,2) circle (2pt) node[anchor=east]{};
\filldraw[black] (1,0) circle (2pt) node[anchor=west]{};
\filldraw[black] (1,.5) circle (2pt) node[anchor=west]{};
\filldraw[black] (1,1) circle (2pt) node[anchor=west]{};
\filldraw[black] (1,1.5) circle (2pt) node[anchor=west]{};
\filldraw[black] (1,2) circle (2pt) node[anchor=west]{};
\draw (0,0) edge [-, bend left=0] (1,0);
\draw (0,.5) edge [-, bend right=20] (0,1);
\draw (1,.5) edge [-, bend left=20] (1,1);
\draw (0,1.5) edge [-, bend right=20] (0,2);
\draw (1,1.5) edge [-, bend left=20] (1,2);
\draw [dashed] (-.12,-.12) rectangle (1.12,2.12);
\node at (.5,-.3) {$h_1h_3$};
\begin{scope}
\clip (0.25,2.115) rectangle (0.75,1.9);
\draw[thick] (0.5,2.14) circle(0.1);
\end{scope}

\filldraw[black] (2,0) circle (2pt) node[anchor=east]{};
\filldraw[black] (2,.5) circle (2pt) node[anchor=east]{};
\filldraw[black] (2,1) circle (2pt) node[anchor=east]{};
\filldraw[black] (2,1.5) circle (2pt) node[anchor=east]{};
\filldraw[black] (2,2) circle (2pt) node[anchor=east]{};
\filldraw[black] (3,0) circle (2pt) node[anchor=west]{};
\filldraw[black] (3,.5) circle (2pt) node[anchor=west]{};
\filldraw[black] (3,1) circle (2pt) node[anchor=west]{};
\filldraw[black] (3,1.5) circle (2pt) node[anchor=west]{};
\filldraw[black] (3,2) circle (2pt) node[anchor=west]{};
\draw (2,0) edge [-, bend left=0] (3,0);
\draw (2,.5) edge [-, bend right=20] (2,1);
\draw (3,.5) edge [-, bend left=30] (3,2);
\draw (2,1.5) edge [-, bend right=20] (2,2);
\draw (3,1.5) edge [-, bend right=20] (3,1);
\draw [dashed] (1.88,-.12) rectangle (3.12,2.12);
\node at (2.5,-.3) {$h_1h_3h_2 $};
\begin{scope}
\clip (2.25,2.115) rectangle (2.75,1.9);
\draw[thick] (2.5,2.14) circle(0.1);
\end{scope}

\filldraw[black] (4,0) circle (2pt) node[anchor=east]{};
\filldraw[black] (4,.5) circle (2pt) node[anchor=east]{};
\filldraw[black] (4,1) circle (2pt) node[anchor=east]{};
\filldraw[black] (4,1.5) circle (2pt) node[anchor=east]{};
\filldraw[black] (4,2) circle (2pt) node[anchor=east]{};
\filldraw[black] (5,0) circle (2pt) node[anchor=west]{};
\filldraw[black] (5,.5) circle (2pt) node[anchor=west]{};
\filldraw[black] (5,1) circle (2pt) node[anchor=west]{};
\filldraw[black] (5,1.5) circle (2pt) node[anchor=west]{};
\filldraw[black] (5,2) circle (2pt) node[anchor=west]{};
\draw (4,0) edge [-, bend left=0] (5,1);
\draw (4,.5) edge [-, bend right=20] (4,1);
\draw (5,1.5) edge [-, bend left=20] (5,2);
\draw (4,1.5) edge [-, bend right=20] (4,2);
\draw (5,0) edge [-, bend left=20] (5,.5);
\draw [dashed] (3.88,-.12) rectangle (5.12,2.12);
\node at (4.5,-.3) {$h_1h_3h_4$};
\begin{scope}
\clip (4.25,2.115) rectangle (4.75,1.9);
\draw[thick] (4.5,2.14) circle(0.1);
\end{scope}

\filldraw[black] (6,0) circle (2pt) node[anchor=east]{};
\filldraw[black] (6,.5) circle (2pt) node[anchor=east]{};
\filldraw[black] (6,1) circle (2pt) node[anchor=east]{};
\filldraw[black] (6,1.5) circle (2pt) node[anchor=east]{};
\filldraw[black] (6,2) circle (2pt) node[anchor=east]{};
\filldraw[black] (7,0) circle (2pt) node[anchor=west]{};
\filldraw[black] (7,.5) circle (2pt) node[anchor=west]{};
\filldraw[black] (7,1) circle (2pt) node[anchor=west]{};
\filldraw[black] (7,1.5) circle (2pt) node[anchor=west]{};
\filldraw[black] (7,2) circle (2pt) node[anchor=west]{};
\draw (6,0) edge [-, bend left=0] (7,2);
\draw (6,.5) edge [-, bend right=20] (6,1);
\draw (7,1) edge [-, bend left=20] (7,1.5);
\draw (6,1.5) edge [-, bend right=20] (6,2);
\draw (7,0) edge [-, bend left=20] (7,.5);
\draw [dashed] (5.88,-.12) rectangle (7.12,2.12);
\node at (6.5,-.3) {$h_1h_3h_2h_4$};
\begin{scope}
\clip (6.25,2.115) rectangle (6.75,1.9);
\draw[thick] (6.5,2.14) circle(0.1);
\end{scope}

\filldraw[black] (8,0) circle (2pt) node[anchor=east]{};
\filldraw[black] (8,.5) circle (2pt) node[anchor=east]{};
\filldraw[black] (8,1) circle (2pt) node[anchor=east]{};
\filldraw[black] (8,1.5) circle (2pt) node[anchor=east]{};
\filldraw[black] (8,2) circle (2pt) node[anchor=east]{};
\filldraw[black] (9,0) circle (2pt) node[anchor=west]{};
\filldraw[black] (9,.5) circle (2pt) node[anchor=west]{};
\filldraw[black] (9,1) circle (2pt) node[anchor=west]{};
\filldraw[black] (9,1.5) circle (2pt) node[anchor=west]{};
\filldraw[black] (9,2) circle (2pt) node[anchor=west]{};
\draw (8,0) edge [-, bend left=0] (9,2);
\draw (8,.5) edge [-, bend right=20] (8,1);
\draw (9,1.5) edge [-, bend right=30] (9,0);
\draw (8,1.5) edge [-, bend right=20] (8,2);
\draw (9,1) edge [-, bend right=20] (9,.5);
\draw [dashed] (7.88,-.12) rectangle (9.12,2.12);
\node at (8.5,-.3) {$h_1h_3h_2h_4h_3$};
\begin{scope}
\clip (8.25,2.115) rectangle (8.75,1.9);
\draw[thick] (8.5,2.14) circle(0.1);
\end{scope}

\filldraw[black] (0,-3) circle (2pt) node[anchor=east]{};
\filldraw[black] (0,-2.5) circle (2pt) node[anchor=east]{};
\filldraw[black] (0,-2) circle (2pt) node[anchor=east]{};
\filldraw[black] (0,-1.5) circle (2pt) node[anchor=east]{};
\filldraw[black] (0,-1) circle (2pt) node[anchor=east]{};
\filldraw[black] (1,-3) circle (2pt) node[anchor=west]{};
\filldraw[black] (1,-2.5) circle (2pt) node[anchor=west]{};
\filldraw[black] (1,-2) circle (2pt) node[anchor=west]{};
\filldraw[black] (1,-1.5) circle (2pt) node[anchor=west]{};
\filldraw[black] (1,-1) circle (2pt) node[anchor=west]{};
\draw (0,-3) edge [-, bend left=0] (1,-3);
\draw (0,-1.5) edge [-, bend left=20] (0,-2);
\draw (1,-2.5) edge [-, bend left=20] (1,-2);
\draw (0,-2.5) edge [-, bend right=30] (0,-1);
\draw (1,-1.5) edge [-, bend left=20] (1,-1);
\draw [dashed] (-.12,-3.12) rectangle (1.12,-0.88);
\node at (.5,-3.3) {$h_2h_1h_3$};
\begin{scope}
\clip (0.25,-0.885) rectangle (0.75,-1.1);
\draw[thick] (0.5,-.86) circle(0.1);
\end{scope}

\filldraw[black] (2,-3) circle (2pt) node[anchor=east]{};
\filldraw[black] (2,-2.5) circle (2pt) node[anchor=east]{};
\filldraw[black] (2,-2) circle (2pt) node[anchor=east]{};
\filldraw[black] (2,-1.5) circle (2pt) node[anchor=east]{};
\filldraw[black] (2,-1) circle (2pt) node[anchor=east]{};
\filldraw[black] (3,-3) circle (2pt) node[anchor=west]{};
\filldraw[black] (3,-2.5) circle (2pt) node[anchor=west]{};
\filldraw[black] (3,-2) circle (2pt) node[anchor=west]{};
\filldraw[black] (3,-1.5) circle (2pt) node[anchor=west]{};
\filldraw[black] (3,-1) circle (2pt) node[anchor=west]{};
\draw (2,-3) edge [-, bend left=0] (3,-3);
\draw (2,-1.5) edge [-, bend left=20] (2,-2);
\draw (3,-2.5) edge [-, bend left=30] (3,-1);
\draw (2,-2.5) edge [-, bend right=30] (2,-1);
\draw (3,-1.5) edge [-, bend right=20] (3,-2);
\draw [dashed] (1.88,-3.12) rectangle (3.12,-.88);
\node at (2.5,-3.3) {$h_2h_1h_3h_2$};
\begin{scope}
\clip (2.25,-0.885) rectangle (2.75,-1.1);
\draw[thick] (2.5,-.86) circle(0.1);
\end{scope}

\filldraw[black] (4,-3) circle (2pt) node[anchor=east]{};
\filldraw[black] (4,-2.5) circle (2pt) node[anchor=east]{};
\filldraw[black] (4,-2) circle (2pt) node[anchor=east]{};
\filldraw[black] (4,-1.5) circle (2pt) node[anchor=east]{};
\filldraw[black] (4,-1) circle (2pt) node[anchor=east]{};
\filldraw[black] (5,-3) circle (2pt) node[anchor=west]{};
\filldraw[black] (5,-2.5) circle (2pt) node[anchor=west]{};
\filldraw[black] (5,-2) circle (2pt) node[anchor=west]{};
\filldraw[black] (5,-1.5) circle (2pt) node[anchor=west]{};
\filldraw[black] (5,-1) circle (2pt) node[anchor=west]{};
\draw (4,-3) edge [-, bend left=0] (5,-2);
\draw (4,-1.5) edge [-, bend left=20] (4,-2);
\draw (5,-1.5) edge [-, bend left=20] (5,-1);
\draw (4,-2.5) edge [-, bend right=30] (4,-1);
\draw (5,-3) edge [-, bend left=20] (5,-2.5);
\draw [dashed] (3.88,-3.12) rectangle (5.12,-.88);
\node at (4.5,-3.3) {$h_2h_1h_3h_4$};
\begin{scope}
\clip (4.25,-0.885) rectangle (4.75,-1.1);
\draw[thick] (4.5,-.86) circle(0.1);
\end{scope}

\filldraw[black] (6,-3) circle (2pt) node[anchor=east]{};
\filldraw[black] (6,-2.5) circle (2pt) node[anchor=east]{};
\filldraw[black] (6,-2) circle (2pt) node[anchor=east]{};
\filldraw[black] (6,-1.5) circle (2pt) node[anchor=east]{};
\filldraw[black] (6,-1) circle (2pt) node[anchor=east]{};
\filldraw[black] (7,-3) circle (2pt) node[anchor=west]{};
\filldraw[black] (7,-2.5) circle (2pt) node[anchor=west]{};
\filldraw[black] (7,-2) circle (2pt) node[anchor=west]{};
\filldraw[black] (7,-1.5) circle (2pt) node[anchor=west]{};
\filldraw[black] (7,-1) circle (2pt) node[anchor=west]{};
\draw (6,-3) edge [-, bend left=0] (7,-1);
\draw (6,-1.5) edge [-, bend left=20] (6,-2);
\draw (7,-2) edge [-, bend left=20] (7,-1.5);
\draw (6,-2.5) edge [-, bend right=30] (6,-1);
\draw (7,-3) edge [-, bend left=20] (7,-2.5);
\draw [dashed] (5.88,-3.12) rectangle (7.12,-.88);
\node at (6.5,-3.3) {$h_2h_1h_3h_2h_4$};
\begin{scope}
\clip (6.25,-0.885) rectangle (6.75,-1.1);
\draw[thick] (6.5,-.86) circle(0.1);
\end{scope}

\filldraw[black] (8,-3) circle (2pt) node[anchor=east]{};
\filldraw[black] (8,-2.5) circle (2pt) node[anchor=east]{};
\filldraw[black] (8,-2) circle (2pt) node[anchor=east]{};
\filldraw[black] (8,-1.5) circle (2pt) node[anchor=east]{};
\filldraw[black] (8,-1) circle (2pt) node[anchor=east]{};
\filldraw[black] (9,-3) circle (2pt) node[anchor=west]{};
\filldraw[black] (9,-2.5) circle (2pt) node[anchor=west]{};
\filldraw[black] (9,-2) circle (2pt) node[anchor=west]{};
\filldraw[black] (9,-1.5) circle (2pt) node[anchor=west]{};
\filldraw[black] (9,-1) circle (2pt) node[anchor=west]{};
\draw (8,-3) edge [-, bend left=0] (9,-1);
\draw (8,-1.5) edge [-, bend left=20] (8,-2);
\draw (9,-1.5) edge [-, bend right=30] (9,-3);
\draw (8,-2.5) edge [-, bend right=30] (8,-1);
\draw (9,-2) edge [-, bend right=20] (9,-2.5);
\draw [dashed] (7.88,-3.12) rectangle (9.12,-.88);
\node at (8.5,-3.3) {$h_2h_1h_3h_2h_4h_3$};
\begin{scope}
\clip (8.25,-0.885) rectangle (8.75,-1.1);
\draw[thick] (8.5,-.86) circle(0.1);
\end{scope}

\filldraw[black] (0,-6) circle (2pt) node[anchor=east]{};
\filldraw[black] (0,-5.5) circle (2pt) node[anchor=east]{};
\filldraw[black] (0,-5) circle (2pt) node[anchor=east]{};
\filldraw[black] (0,-4.5) circle (2pt) node[anchor=east]{};
\filldraw[black] (0,-4) circle (2pt) node[anchor=east]{};
\filldraw[black] (1,-6) circle (2pt) node[anchor=west]{};
\filldraw[black] (1,-5.5) circle (2pt) node[anchor=west]{};
\filldraw[black] (1,-5) circle (2pt) node[anchor=west]{};
\filldraw[black] (1,-4.5) circle (2pt) node[anchor=west]{};
\filldraw[black] (1,-4) circle (2pt) node[anchor=west]{};
\draw (0,-5) edge [-, bend left=0] (1,-6);
\draw (0,-5.5) edge [-, bend left=20] (0,-6);
\draw (1,-5.5) edge [-, bend left=20] (1,-5);
\draw (0,-4.5) edge [-, bend right=20] (0,-4);
\draw (1,-4.5) edge [-, bend left=20] (1,-4);
\draw [dashed] (-.12,-6.12) rectangle (1.12,-3.88);
\node at (.5,-6.3) {$h_1h_4h_3$};
\begin{scope}
\clip (0.25,-3.885) rectangle (0.75,-4.1);
\draw[thick] (0.5,-3.86) circle(0.1);
\end{scope}

\filldraw[black] (2,-6) circle (2pt) node[anchor=east]{};
\filldraw[black] (2,-5.5) circle (2pt) node[anchor=east]{};
\filldraw[black] (2,-5) circle (2pt) node[anchor=east]{};
\filldraw[black] (2,-4.5) circle (2pt) node[anchor=east]{};
\filldraw[black] (2,-4) circle (2pt) node[anchor=east]{};
\filldraw[black] (3,-6) circle (2pt) node[anchor=west]{};
\filldraw[black] (3,-5.5) circle (2pt) node[anchor=west]{};
\filldraw[black] (3,-5) circle (2pt) node[anchor=west]{};
\filldraw[black] (3,-4.5) circle (2pt) node[anchor=west]{};
\filldraw[black] (3,-4) circle (2pt) node[anchor=west]{};
\draw (2,-5) edge [-, bend left=0] (3,-6);
\draw (2,-5.5) edge [-, bend left=20] (2,-6);
\draw (3,-5.5) edge [-, bend left=30] (3,-4);
\draw (2,-4.5) edge [-, bend right=20] (2,-4);
\draw (3,-4.5) edge [-, bend right=20] (3,-5);
\draw [dashed] (1.88,-6.12) rectangle (3.12,-3.88);
\node at (2.5,-6.3) {$h_1h_4h_3h_2$};
\begin{scope}
\clip (2.25,-3.885) rectangle (2.75,-4.1);
\draw[thick] (2.5,-3.86) circle(0.1);
\end{scope}

\filldraw[black] (4,-6) circle (2pt) node[anchor=east]{};
\filldraw[black] (4,-5.5) circle (2pt) node[anchor=east]{};
\filldraw[black] (4,-5) circle (2pt) node[anchor=east]{};
\filldraw[black] (4,-4.5) circle (2pt) node[anchor=east]{};
\filldraw[black] (4,-4) circle (2pt) node[anchor=east]{};
\filldraw[black] (5,-6) circle (2pt) node[anchor=west]{};
\filldraw[black] (5,-5.5) circle (2pt) node[anchor=west]{};
\filldraw[black] (5,-5) circle (2pt) node[anchor=west]{};
\filldraw[black] (5,-4.5) circle (2pt) node[anchor=west]{};
\filldraw[black] (5,-4) circle (2pt) node[anchor=west]{};
\draw (4,-5) edge [-, bend left=0] (5,-5);
\draw (4,-5.5) edge [-, bend left=20] (4,-6);
\draw (5,-4.5) edge [-, bend left=20] (5,-4);
\draw (4,-4.5) edge [-, bend right=20] (4,-4);
\draw (5,-6) edge [-, bend left=20] (5,-5.5);
\draw [dashed] (3.88,-6.12) rectangle (5.12,-3.88);
\node at (4.5,-6.3) {$h_1h_4$};
\begin{scope}
\clip (4.25,-3.885) rectangle (4.75,-4.1);
\draw[thick] (4.5,-3.86) circle(0.1);
\end{scope}

\filldraw[black] (6,-6) circle (2pt) node[anchor=east]{};
\filldraw[black] (6,-5.5) circle (2pt) node[anchor=east]{};
\filldraw[black] (6,-5) circle (2pt) node[anchor=east]{};
\filldraw[black] (6,-4.5) circle (2pt) node[anchor=east]{};
\filldraw[black] (6,-4) circle (2pt) node[anchor=east]{};
\filldraw[black] (7,-6) circle (2pt) node[anchor=west]{};
\filldraw[black] (7,-5.5) circle (2pt) node[anchor=west]{};
\filldraw[black] (7,-5) circle (2pt) node[anchor=west]{};
\filldraw[black] (7,-4.5) circle (2pt) node[anchor=west]{};
\filldraw[black] (7,-4) circle (2pt) node[anchor=west]{};
\draw (6,-5) edge [-, bend left=0] (7,-4);
\draw (6,-5.5) edge [-, bend left=20] (6,-6);
\draw (7,-5) edge [-, bend left=20] (7,-4.5);
\draw (6,-4.5) edge [-, bend right=20] (6,-4);
\draw (7,-6) edge [-, bend left=20] (7,-5.5);
\draw [dashed] (5.88,-6.12) rectangle (7.12,-3.88);
\node at (6.5,-6.3) {$h_1h_2h_4$};
\begin{scope}
\clip (6.25,-3.885) rectangle (6.75,-4.1);
\draw[thick] (6.5,-3.86) circle(0.1);
\end{scope}

\filldraw[black] (8,-6) circle (2pt) node[anchor=east]{};
\filldraw[black] (8,-5.5) circle (2pt) node[anchor=east]{};
\filldraw[black] (8,-5) circle (2pt) node[anchor=east]{};
\filldraw[black] (8,-4.5) circle (2pt) node[anchor=east]{};
\filldraw[black] (8,-4) circle (2pt) node[anchor=east]{};
\filldraw[black] (9,-6) circle (2pt) node[anchor=west]{};
\filldraw[black] (9,-5.5) circle (2pt) node[anchor=west]{};
\filldraw[black] (9,-5) circle (2pt) node[anchor=west]{};
\filldraw[black] (9,-4.5) circle (2pt) node[anchor=west]{};
\filldraw[black] (9,-4) circle (2pt) node[anchor=west]{};
\draw (8,-5) edge [-, bend left=0] (9,-4);
\draw (8,-5.5) edge [-, bend left=20] (8,-6);
\draw (9,-4.5) edge [-, bend right=30] (9,-6);
\draw (8,-4.5) edge [-, bend right=20] (8,-4);
\draw (9,-5) edge [-, bend right=20] (9,-5.5);
\draw [dashed] (7.88,-6.12) rectangle (9.12,-3.88);
\node at (8.5,-6.3) {$h_1h_2h_4h_3$};
\begin{scope}
\clip (8.25,-3.885) rectangle (8.75,-4.1);
\draw[thick] (8.5,-3.86) circle(0.1);
\end{scope}

\filldraw[black] (0,-9) circle (2pt) node[anchor=east]{};
\filldraw[black] (0,-8.5) circle (2pt) node[anchor=east]{};
\filldraw[black] (0,-8) circle (2pt) node[anchor=east]{};
\filldraw[black] (0,-7.5) circle (2pt) node[anchor=east]{};
\filldraw[black] (0,-7) circle (2pt) node[anchor=east]{};
\filldraw[black] (1,-9) circle (2pt) node[anchor=west]{};
\filldraw[black] (1,-8.5) circle (2pt) node[anchor=west]{};
\filldraw[black] (1,-8) circle (2pt) node[anchor=west]{};
\filldraw[black] (1,-7.5) circle (2pt) node[anchor=west]{};
\filldraw[black] (1,-7) circle (2pt) node[anchor=west]{};
\draw (0,-7) edge [-, bend left=0] (1,-9);
\draw (0,-8.5) edge [-, bend left=20] (0,-9);
\draw (1,-8.5) edge [-, bend left=20] (1,-8);
\draw (0,-7.5) edge [-, bend left=20] (0,-8);
\draw (1,-7.5) edge [-, bend left=20] (1,-7);
\draw [dashed] (-.12,-9.12) rectangle (1.12,-6.88);
\node at (.5,-9.3) {$h_2h_1h_4h_3$};
\begin{scope}
\clip (0.25,-6.885) rectangle (0.75,-7.1);
\draw[thick] (0.5,-6.86) circle(0.1);
\end{scope}

\filldraw[black] (2,-9) circle (2pt) node[anchor=east]{};
\filldraw[black] (2,-8.5) circle (2pt) node[anchor=east]{};
\filldraw[black] (2,-8) circle (2pt) node[anchor=east]{};
\filldraw[black] (2,-7.5) circle (2pt) node[anchor=east]{};
\filldraw[black] (2,-7) circle (2pt) node[anchor=east]{};
\filldraw[black] (3,-9) circle (2pt) node[anchor=west]{};
\filldraw[black] (3,-8.5) circle (2pt) node[anchor=west]{};
\filldraw[black] (3,-8) circle (2pt) node[anchor=west]{};
\filldraw[black] (3,-7.5) circle (2pt) node[anchor=west]{};
\filldraw[black] (3,-7) circle (2pt) node[anchor=west]{};
\draw (2,-7) edge [-, bend left=0] (3,-9);
\draw (2,-8.5) edge [-, bend left=20] (2,-9);
\draw (3,-8.5) edge [-, bend left=30] (3,-7);
\draw (2,-7.5) edge [-, bend left=20] (2,-8);
\draw (3,-7.5) edge [-, bend right=20] (3,-8);
\draw [dashed] (1.88,-9.12) rectangle (3.12,-6.88);
\node at (2.5,-9.3) {$h_2h_1h_4h_3h_2$};
\begin{scope}
\clip (2.25,-6.885) rectangle (2.75,-7.1);
\draw[thick] (2.5,-6.86) circle(0.1);
\end{scope}

\filldraw[black] (4,-9) circle (2pt) node[anchor=east]{};
\filldraw[black] (4,-8.5) circle (2pt) node[anchor=east]{};
\filldraw[black] (4,-8) circle (2pt) node[anchor=east]{};
\filldraw[black] (4,-7.5) circle (2pt) node[anchor=east]{};
\filldraw[black] (4,-7) circle (2pt) node[anchor=east]{};
\filldraw[black] (5,-9) circle (2pt) node[anchor=west]{};
\filldraw[black] (5,-8.5) circle (2pt) node[anchor=west]{};
\filldraw[black] (5,-8) circle (2pt) node[anchor=west]{};
\filldraw[black] (5,-7.5) circle (2pt) node[anchor=west]{};
\filldraw[black] (5,-7) circle (2pt) node[anchor=west]{};
\draw (4,-7) edge [-, bend left=0] (5,-8);
\draw (4,-8.5) edge [-, bend left=20] (4,-9);
\draw (5,-7.5) edge [-, bend left=20] (5,-7);
\draw (4,-7.5) edge [-, bend left=20] (4,-8);
\draw (5,-9) edge [-, bend left=20] (5,-8.5);
\draw [dashed] (3.88,-9.12) rectangle (5.12,-6.88);
\node at (4.5,-9.3) {$h_2h_1h_4$};
\begin{scope}
\clip (4.25,-6.885) rectangle (4.75,-7.1);
\draw[thick] (4.5,-6.86) circle(0.1);
\end{scope}

\filldraw[black] (6,-9) circle (2pt) node[anchor=east]{};
\filldraw[black] (6,-8.5) circle (2pt) node[anchor=east]{};
\filldraw[black] (6,-8) circle (2pt) node[anchor=east]{};
\filldraw[black] (6,-7.5) circle (2pt) node[anchor=east]{};
\filldraw[black] (6,-7) circle (2pt) node[anchor=east]{};
\filldraw[black] (7,-9) circle (2pt) node[anchor=west]{};
\filldraw[black] (7,-8.5) circle (2pt) node[anchor=west]{};
\filldraw[black] (7,-8) circle (2pt) node[anchor=west]{};
\filldraw[black] (7,-7.5) circle (2pt) node[anchor=west]{};
\filldraw[black] (7,-7) circle (2pt) node[anchor=west]{};
\draw (6,-7) edge [-, bend left=0] (7,-7);
\draw (6,-8.5) edge [-, bend left=20] (6,-9);
\draw (7,-8) edge [-, bend left=20] (7,-7.5);
\draw (6,-7.5) edge [-, bend left=20] (6,-8);
\draw (7,-9) edge [-, bend left=20] (7,-8.5);
\draw [dashed] (5.88,-9.12) rectangle (7.12,-6.88);
\node at (6.5,-9.3) {$h_2h_4$};
\begin{scope}
\clip (6.25,-6.885) rectangle (6.75,-7.1);
\draw[thick] (6.5,-6.86) circle(0.1);
\end{scope}

\filldraw[black] (8,-9) circle (2pt) node[anchor=east]{};
\filldraw[black] (8,-8.5) circle (2pt) node[anchor=east]{};
\filldraw[black] (8,-8) circle (2pt) node[anchor=east]{};
\filldraw[black] (8,-7.5) circle (2pt) node[anchor=east]{};
\filldraw[black] (8,-7) circle (2pt) node[anchor=east]{};
\filldraw[black] (9,-9) circle (2pt) node[anchor=west]{};
\filldraw[black] (9,-8.5) circle (2pt) node[anchor=west]{};
\filldraw[black] (9,-8) circle (2pt) node[anchor=west]{};
\filldraw[black] (9,-7.5) circle (2pt) node[anchor=west]{};
\filldraw[black] (9,-7) circle (2pt) node[anchor=west]{};
\draw (8,-7) edge [-, bend left=0] (9,-7);
\draw (8,-8.5) edge [-, bend left=20] (8,-9);
\draw (9,-7.5) edge [-, bend right=30] (9,-9);
\draw (8,-7.5) edge [-, bend left=20] (8,-8);
\draw (9,-8) edge [-, bend right=20] (9,-8.5);
\draw [dashed] (7.88,-9.12) rectangle (9.12,-6.88);
\node at (8.5,-9.3) {$h_2h_4h_3$};
\begin{scope}
\clip (8.25,-6.885) rectangle (8.75,-7.1);
\draw[thick] (8.5,-6.86) circle(0.1);
\end{scope}

\filldraw[black] (0,-12) circle (2pt) node[anchor=east]{};
\filldraw[black] (0,-11.5) circle (2pt) node[anchor=east]{};
\filldraw[black] (0,-11) circle (2pt) node[anchor=east]{};
\filldraw[black] (0,-10.5) circle (2pt) node[anchor=east]{};
\filldraw[black] (0,-10) circle (2pt) node[anchor=east]{};
\filldraw[black] (1,-12) circle (2pt) node[anchor=west]{};
\filldraw[black] (1,-11.5) circle (2pt) node[anchor=west]{};
\filldraw[black] (1,-11) circle (2pt) node[anchor=west]{};
\filldraw[black] (1,-10.5) circle (2pt) node[anchor=west]{};
\filldraw[black] (1,-10) circle (2pt) node[anchor=west]{};
\draw (0,-10) edge [-, bend left=0] (1,-12);
\draw (0,-10.5) edge [-, bend left=30] (0,-12);
\draw (1,-11.5) edge [-, bend left=20] (1,-11);
\draw (0,-11.5) edge [-, bend right=20] (0,-11);
\draw (1,-10.5) edge [-, bend left=20] (1,-10);
\draw [dashed] (-.12,-12.12) rectangle (1.12,-9.88);
\node at (.5,-12.3) {$h_3h_2h_1h_4h_3$};
\begin{scope}
\clip (0.25,-9.885) rectangle (0.75,-10.1);
\draw[thick] (0.5,-9.86) circle(0.1);
\end{scope}

\filldraw[black] (2,-12) circle (2pt) node[anchor=east]{};
\filldraw[black] (2,-11.5) circle (2pt) node[anchor=east]{};
\filldraw[black] (2,-11) circle (2pt) node[anchor=east]{};
\filldraw[black] (2,-10.5) circle (2pt) node[anchor=east]{};
\filldraw[black] (2,-10) circle (2pt) node[anchor=east]{};
\filldraw[black] (3,-12) circle (2pt) node[anchor=west]{};
\filldraw[black] (3,-11.5) circle (2pt) node[anchor=west]{};
\filldraw[black] (3,-11) circle (2pt) node[anchor=west]{};
\filldraw[black] (3,-10.5) circle (2pt) node[anchor=west]{};
\filldraw[black] (3,-10) circle (2pt) node[anchor=west]{};
\draw (2,-10) edge [-, bend left=0] (3,-12);
\draw (2,-10.5) edge [-, bend left=30] (2,-12);
\draw (3,-11.5) edge [-, bend left=30] (3,-10);
\draw (2,-11.5) edge [-, bend right=20] (2,-11);
\draw (3,-10.5) edge [-, bend right=20] (3,-11);
\draw [dashed] (1.88,-12.12) rectangle (3.12,-9.88);
\node at (2.5,-12.3) {$h_3h_2h_1h_4h_3h_2$};
\begin{scope}
\clip (2.25,-9.885) rectangle (2.75,-10.1);
\draw[thick] (2.5,-9.86) circle(0.1);
\end{scope}

\filldraw[black] (4,-12) circle (2pt) node[anchor=east]{};
\filldraw[black] (4,-11.5) circle (2pt) node[anchor=east]{};
\filldraw[black] (4,-11) circle (2pt) node[anchor=east]{};
\filldraw[black] (4,-10.5) circle (2pt) node[anchor=east]{};
\filldraw[black] (4,-10) circle (2pt) node[anchor=east]{};
\filldraw[black] (5,-12) circle (2pt) node[anchor=west]{};
\filldraw[black] (5,-11.5) circle (2pt) node[anchor=west]{};
\filldraw[black] (5,-11) circle (2pt) node[anchor=west]{};
\filldraw[black] (5,-10.5) circle (2pt) node[anchor=west]{};
\filldraw[black] (5,-10) circle (2pt) node[anchor=west]{};
\draw (4,-10) edge [-, bend left=0] (5,-11);
\draw (4,-10.5) edge [-, bend left=30] (4,-12);
\draw (5,-10.5) edge [-, bend left=20] (5,-10);
\draw (4,-11.5) edge [-, bend right=20] (4,-11);
\draw (5,-12) edge [-, bend left=20] (5,-11.5);
\draw [dashed] (3.88,-12.12) rectangle (5.12,-9.88);
\node at (4.5,-12.3) {$h_3h_2h_1h_4$};
\begin{scope}
\clip (4.25,-9.885) rectangle (4.75,-10.1);
\draw[thick] (4.5,-9.86) circle(0.1);
\end{scope}

\filldraw[black] (6,-12) circle (2pt) node[anchor=east]{};
\filldraw[black] (6,-11.5) circle (2pt) node[anchor=east]{};
\filldraw[black] (6,-11) circle (2pt) node[anchor=east]{};
\filldraw[black] (6,-10.5) circle (2pt) node[anchor=east]{};
\filldraw[black] (6,-10) circle (2pt) node[anchor=east]{};
\filldraw[black] (7,-12) circle (2pt) node[anchor=west]{};
\filldraw[black] (7,-11.5) circle (2pt) node[anchor=west]{};
\filldraw[black] (7,-11) circle (2pt) node[anchor=west]{};
\filldraw[black] (7,-10.5) circle (2pt) node[anchor=west]{};
\filldraw[black] (7,-10) circle (2pt) node[anchor=west]{};
\draw (6,-10) edge [-, bend left=0] (7,-10);
\draw (6,-10.5) edge [-, bend left=30] (6,-12);
\draw (7,-11) edge [-, bend left=20] (7,-10.5);
\draw (6,-11.5) edge [-, bend right=20] (6,-11);
\draw (7,-12) edge [-, bend left=20] (7,-11.5);
\draw [dashed] (5.88,-12.12) rectangle (7.12,-9.88);
\node at (6.5,-12.3) {$h_3h_2h_4$};
\begin{scope}
\clip (6.25,-9.885) rectangle (6.75,-10.1);
\draw[thick] (6.5,-9.86) circle(0.1);
\end{scope}

\filldraw[black] (8,-12) circle (2pt) node[anchor=east]{};
\filldraw[black] (8,-11.5) circle (2pt) node[anchor=east]{};
\filldraw[black] (8,-11) circle (2pt) node[anchor=east]{};
\filldraw[black] (8,-10.5) circle (2pt) node[anchor=east]{};
\filldraw[black] (8,-10) circle (2pt) node[anchor=east]{};
\filldraw[black] (9,-12) circle (2pt) node[anchor=west]{};
\filldraw[black] (9,-11.5) circle (2pt) node[anchor=west]{};
\filldraw[black] (9,-11) circle (2pt) node[anchor=west]{};
\filldraw[black] (9,-10.5) circle (2pt) node[anchor=west]{};
\filldraw[black] (9,-10) circle (2pt) node[anchor=west]{};
\draw (8,-10) edge [-, bend left=0] (9,-10);
\draw (8,-10.5) edge [-, bend left=30] (8,-12);
\draw (9,-10.5) edge [-, bend right=30] (9,-12);
\draw (8,-11.5) edge [-, bend right=20] (8,-11);
\draw (9,-11) edge [-, bend right=20] (9,-11.5);
\draw [dashed] (7.88,-12.12) rectangle (9.12,-9.88);
\node at (8.5,-12.3) {$h_3h_2h_4h_3$};
\begin{scope}
\clip (8.25,-9.885) rectangle (8.75,-10.1);
\draw[thick] (8.5,-9.86) circle(0.1);
\end{scope}
\end{tikzpicture}
\end{center}
    \caption{The elements of the $\J$-class $A_3$}
     \label{fig:A3}
\end{figure}
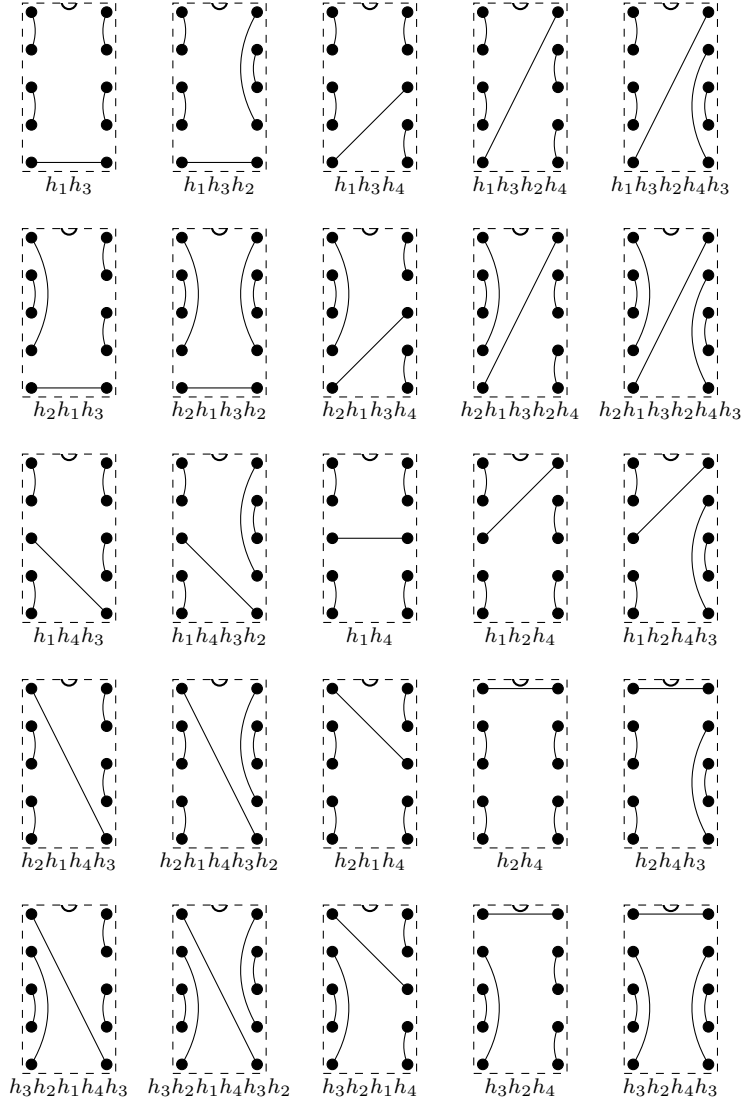

\begin{lem}\label{R_non-eq}
Let $2< n_1,n_2$ and let $a=a_1\cdots a_{n_1}$ and $b=b_1\cdots b_{n_2}$, for some elements $a_1,\ldots,a_{n_1},b_1,\ldots,b_{n_2}\in \{h_1,h_2,h_3,h_4\}$, such that $a_1\cdots a_{n_1-1}\in A_2$, $b_1\cdots b_{n_2-1}\in A_2$ and $a_{1}=b_{1}$.
If one of the following conditions hold:
\begin{enumerate}
\item[$(\ast1)$] $\{a_{n_1-1},a_{n_1}\}=\{h_1,h_3\}$ and $\{b_{n_2-1},b_{n_2}\}=\{h_2,h_4\}$;
\item[$(\ast2)$] $a_{n_1-1}=b_{n_2-1}=h_4$ and $\{a_{n_1},b_{n_2}\}=\{h_1,h_2\}$;
\item[$(\ast3)$] $a_{n_1-1}=b_{n_2-1}=h_1$ and $\{a_{n_1},b_{n_2}\}=\{h_3,h_4\}$;
\item[$(\ast4)$] $a_{n_1-1}=b_{n_2}=h_4$, $a_{n_1}=h_1$ and $b_{n_2-1}=h_2$;
\item[$(\ast5)$] $a_{n_1-1}=b_{n_2}=h_1$, $a_{n_1}=h_4$ and $b_{n_2-1}=h_3$,
\end{enumerate}
then the elements $a$ and $b$ are not in the same $\R$-class.

Also, if one of the following conditions hold:
\begin{enumerate}
\item[$(\#1)$] $\{a_{n_1-1},b_{n_2-1}\}=\{h_1,h_2\}$ and $a_{n_1}=b_{n_2}=h_4$;
\item[$(\#2)$] $\{a_{n_1-1},b_{n_2-1}\}=\{h_3,h_4\}$ and $a_{n_1}=b_{n_2}=h_1$;
\item[$(\#3)$] $a_{n_1-1}=b_{n_2}=h_4$, $a_{n_1}=h_2$ and $b_{n_2-1}=h_1$;
\item[$(\#4)$] $a_{n_1-1}=b_{n_2}=h_1$, $a_{n_1}=h_3$ and $b_{n_2-1}=h_4$,
\end{enumerate}
then the elements $a$ and $b$ are in the same $\R$-class.
\end{lem}

\begin{proof}
\noindent $(\ast1)\colon$
We have the following cases subject to the element $a_{1}$:
\begin{enumerate}
\item $(a_{1}=h_1)\colon$
since $h_1h_2h_3h_1=h_1h_2h_1h_3=h_1h_3$, we have $a=h_1h_3$. Also, as $h_1h_2h_3h_4h_2=h_1h_2h_3h_2h_4=h_1h_2h_4$, we have $b=h_1h_2h_4$. Now, as $h_1h_3$ and $h_1h_2h_4$ are not in the same $\R$-class, the elements $a$ and $b$ are not in the same $\R$-class.
\item $(a_{1}=h_2)\colon$ 
we have $a=h_2h_3h_1$. Also, as $h_2h_3h_4h_2=h_2h_3h_2h_4=h_2h_4$, we have $b=h_2h_4$. Now, as $h_2h_3h_1$ and $h_2h_4$ are not in the same $\R$-class, the result follows.
\item $(a_{1}=h_3)\colon$
since $h_3h_2h_1h_3=h_3h_2h_3h_1=h_1h_3$, we have $a=h_1h_3$. Also, we have $b=h_3h_2h_4$. Now, as $h_1h_3$ and $h_3h_2h_4$ are not in the same $\R$-class, the result follows.
\item $(a_{1}=h_4)\colon$
since $h_4h_3h_2h_1h_3=h_4h_3h_2h_3h_1=h_4h_3h_1=h_4h_3h_1$, we have $a=h_4h_3h_1$. Also, as $h_4h_3h_2h_4=h_4h_3h_4h_2=h_4h_2$, we have $b=h_4h_2$. Now, as $h_4h_3h_1$ and $h_4h_2$ are not in the same $\R$-class, the elements $a$ and $b$ are not in the same $\R$-class.
\end{enumerate}

\noindent $(\ast2)\colon$
If $a_{1}=h_4$, then we have $a_1\cdots a_{n_1-1}=b_1\cdots b_{n_2-1}=h_4$. As, $h_1h_4$ and $h_2h_4$ are not in the same $\R$-class, the elements $a$ and $b$ are not in the same $\R$-class.
If $a_{1}\neq h_4$, then there exist integers $1\leq i_1< n_1-1$ and $1\leq i_2< n_2-1$ such that $a_{i_1}=h_3$, $b_{i_2}=h_3$ and
$a_{i_1+1}\cdots a_{n_1-1}=b_{i_2+1}\cdots b_{n_2-1}=h_4$. Then, we have 
$\{a_{i_1}\cdots a_{n_1},b_{i_2}\cdots b_{n_2}\}=\{h_3h_4h_1,h_3h_4h_2\}=\{h_3h_1h_4,h_3h_4h_2\}$.
Now, by the previous part, the result follows.

\noindent $(\ast3)\colon$ Similarly the previous part, the elements $a$ and $b$ are not in the same $\R$-class.

\noindent $(\ast4)\colon$
we have the following cases subject to the element $a_{1}$:
\begin{enumerate}
\item $(a_{1}=h_1)\colon$
we have $a=h_1h_2h_3h_4h_1=h_1h_2h_3h_1h_4=h_1h_2h_1h_3h_4=h_1h_3h_4$ and $b=h_1h_2h_4$. Now, as $h_1h_3h_4$ and $h_1h_2h_4$ are not in the same $\R$-class, the elements $a$ and $b$ are not in the same $\R$-class.  
\item $(a_{1}=h_2)\colon$ 
we have $a=h_2h_3h_4h_1$ and $b=h_2h_4$. Now, as $h_2h_3h_4h_1$ and $h_2h_4$ are not in the same $\R$-class, the elements $a$ and $b$ are not in the same $\R$-class. 
\item $(a_{1}=h_3)\colon$
we have $a=h_3h_4h_1$ and $b=h_3h_2h_4$. Now, as $h_3h_4h_1$ and $h_3h_2h_4$ are not in the same $\R$-class, the elements $a$ and $b$ are not in the same $\R$-class.  
\item $(a_{1}=h_4)\colon$
we have $a=h_4h_1$ and $b=h_4h_3h_2h_4=h_4h_3h_4h_2=h_4h_2$. Now, as $h_4h_1$ and $h_4h_2$ are not in the same $\R$-class, the elements $a$ and $b$ are not in the same $\R$-class. 
\end{enumerate}

\noindent $(\ast5)\colon$
Similarly Part $(\ast4)$, the elements $a$ and $b$ are not in the same $\R$-class.

\noindent $(\#1)\colon$
By symmetry, we may assume that $a_{n_1-1}=h_1$ and $b_{n_2-1}=h_2$.
We have the following cases subject to the element $a_{1}$:
\begin{enumerate}
\item $(a_{1}=h_1)\colon$
we have $a=h_1h_4$ and $b=h_1h_2h_4$. Now, as $h_1h_4$ and $h_1h_2h_4$ are in the same $\R$-class, the elements $a$ and $b$ are in the same $\R$-class. 
\item $(a_{1}=h_2)\colon$ 
we have $a=h_2h_1h_4$ and $b=h_2h_4$. Now, as $h_2h_1h_4$ and $h_2h_4$ are in the same $\R$-class, the result follows.
\item $(a_{1}=h_3)\colon$
we have $a=h_3h_2h_1h_4$ and $b=h_3h_2h_4$. Now, as $h_3h_2h_1h_4$ and $h_3h_2h_4$ are in the same $\R$-class, the result follows.
\item $(a_{1}=h_4)\colon$
since $h_4h_3h_2h_1h_4=h_4h_3h_2h_4h_1=h_4h_3h_4h_2h_1=h_4h_2h_1$, we have $a=h_4h_2h_1$. Also, as $h_4h_3h_2h_4=h_4h_3h_4h_2=h_4h_2$, we have $b=h_4h_2$. Now, as $h_4h_2h_1$ and $h_4h_2$ are in the same $\R$-class, the elements $a$ and $b$ are in the same $\R$-class.
\end{enumerate}

Similarly Part $(\#1)$, other Parts $(\#2)$, $(\#3)$ and $(\#4)$ hold.
\end{proof}

Similarly Lemma~\ref{R_non-eq}, we have the following lemma.
 
\begin{lem}\label{L_non-eq}
Let $2< n_1,n_2$ and let $a=a_1\cdots a_{n_1}$ and $b=b_1\cdots b_{n_2}$, for some elements $a_1,\ldots,a_{n_1},b_1,\ldots,b_{n_2}\in \{h_1,h_2,h_3,h_4\}$, such that $a_2\cdots a_{n_1}\in A_2$, $b_2\cdots b_{n_2}\in A_2$ and $a_{n_1}=b_{n_2}$.
If one of the following conditions hold:
\begin{enumerate}
\item[$(\ast1)$] $\{a_1,a_2\}=\{h_1,h_3\}$ and $\{b_1,b_2\}=\{h_2,h_4\}$;
\item[$(\ast2)$] $\{a_1,b_1\}=\{h_1,h_2\}$ and $a_2=b_2=h_4$;
\item[$(\ast3)$] $\{a_1,b_1\}=\{h_3,h_4\}$ and $a_2=b_2=h_1$;
\item[$(\ast4)$] $a_1=b_2=h_4$, $a_2=h_2$ and $b_1=h_1$;
\item[$(\ast5)$] $a_1=b_2=h_1$, $a_2=h_3$ and $b_1=h_4$,
\end{enumerate}
then the elements $a$ and $b$ are not in the same $\eL$-class.

Also, if one of the following conditions hold:
\begin{enumerate}
\item[$(\#1)$] $a_{1}=b_{1}=h_4$ and $\{a_{2},b_{2}\}=\{h_1,h_2\}$;
\item[$(\#2)$] $a_{1}=b_{1}=h_1$ and $\{a_{2},b_{2}\}=\{h_3,h_4\}$;
\item[$(\#3)$] $a_1=b_2=h_4$, $a_2=h_1$ and $b_1=h_2$;
\item[$(\#4)$] $a_1=b_2=h_1$, $a_2=h_4$ and $b_1=h_3$,
\end{enumerate}
then the elements $a$ and $b$ are in the same $\eL$-class.
\end{lem}


\section{Characterization of the identities of the monoid $\mathcal{J}_5$}

Let $w$ and $v$ be words of $X^*$.

\begin{lem}\label{c-w-v}
If $w \JJJ v$, then we have $c(w)=c(v)$.
\end{lem}

\begin{proof}
Suppose the contrary, that there exists an letter $x\in c(w)\setminus c(v)$.
Substituting for all letters in $c(v)$ the value $h_1$, we obtain $v=h_1$. Substituting for all letters in $c(w)\setminus \{x\}$ the value $h_1$ and the letter $x$ the value $h_3$ we obtain $w= h_3$ or $w\in A_3$. Hence, $\mathcal{J}_5$ does not satisfy the identity $w= v$. 
Also, if there exists an letter $x\in c(v)\setminus c(w)$, we have a similar contradiction. Thus, we have $c(w)=c(v)$.
\end{proof}

By Lemma~\ref{c-w-v}, 
if $\mathcal{J}_5$ satisfies the identity $w=v$, then
$w=1$ if and only if $v=1$. Then, we assume that $w,v\neq 1$ and $c(w)=c(v)$. 

Let $Y$ be a non empty subset of $c(w)$. By Lemma~\ref{c-w-v}, we have $c(w_Y)= c(v_Y)$.
There exist letters $y_1, \ldots, y_r$, $z_1, \ldots, z_s$ in $Y$ 
such that 
$w_Y= y_1\cdots y_r$ and $v_Y= z_1\cdots z_s$. 

\begin{lem}\label{M-x_1=y_1-x_n=y_m}
If $w_Y\JJJ v_Y$, then we have $y_1=z_1$ and $y_{r}=z_{s}$.
\end{lem}

\begin{proof}
Let $a_1,a_2\in A_3$ such that $a_1$ and $a_2$ are not in the same $\R$-class. Substituting the letter $y_1$ the value $a_1$, the letter $z_1$ the value $a_2$, we obtain that $w$ and $v$ are not in the same $\R$-class. Then, $\mathcal{J}_5$ does not satisfy the identity $w_Y=v_Y$, a contradiction.

Similarly, we have $y_{r}=z_{s}$.
\end{proof}

Throughout the remainder of this section before the main theorem, we assume that the following conditions hold for the subset $Y$ as follows:
\begin{enumerate}
\item $y_1=z_1$;
\item $y_{r}=z_{s}$.
\end{enumerate}

\begin{lem}\label{word-of-length-2}
If $w_Y\JJJ v_Y$, 
then each word of length 2 occurs in $w_Y$ if and only if it occurs in $v_Y$.
\end{lem}

\begin{proof}
We prove the result by contradiction.

Let $xy$ be a word of length 2 occurs in $w_Y$ and does not occur in $v_Y$. 
Since $xy$ occurs in $w_Y$, there exist letters $y_i,y_{i+1}\in Y$, for some integer $1\leq i\leq r-1$ such that $y_i=x$ and $y_{i+1}=y$.
We have two cases as follows:
\begin{enumerate}
\item[$(x\neq y)$:]
If $\abs{c(w_Y)}=2$, 
then there exists an integer $1\leq j<s$ such that $z_1=\cdots=z_j=y$ and $z_{j+1}=\cdots=z_s=x$.
Substituting the letter $y_i$ the value $h_3$ and the letter $y_{i+1}$ the value $h_1h_2$, we obtain that $w_Y\in A_3$, because $h_3h_1h_2\in A_3$. Now, as $v_Y=h_1h_2h_3\in A_2$, there is a contradiction.
Hence, we suppose that $\abs{c(w_Y)}> 2$.
Again, by substituting the letter $y_i$ the value $h_3$, the letter $y_{i+1}$ the value $h_1h_2$ and for all letters in $c(w_Y)\setminus\{y_i,y_{i+1}\}$ the value $h_2$, we again obtain that $w_Y\in A_3$. Also, since $h_2h_3h_2=h_2$, $h_3h_2h_3=h_3$, $zy=h_2h_1h_2=h_2$ for all $z\in c(w_Y)\setminus\{y_i,y_{i+1}\}$ and $xy$ does not occur in $v_Y$, one of the following conditions hold:\\
\noindent $(z_1=x)$: $v_Y=h_3h_2$ or $v_Y=h_3$;\\
\noindent $(z_1=y)$: $v_Y=h_1h_2$ or $v_Y=h_1h_2h_3$;\\
\noindent $(z_1\not\in\{x,y\})$: $v_Y=h_2h_3$ or $v_Y=h_2$.\\
Thus, we have $v_Y\in A_2$, a contradiction.
\item[$(x=y)$:]
First suppose that $\abs{c(w_Y)}=1$.
Substitute the letter $y_i$ the value $h_1h_2h_3$. We have $w_Y\in A_3$ and $v_Y\in A_2$, because $x^2$ does not occur in $v_Y$. A contradiction. Hence, we suppose that $\abs{c(w_Y)}> 1$.
Substituting the letter $y_i$ the value $h_3h_2h_1$ and for all letters in $c(w_Y)\setminus\{y_i\}$ the value $h_2$, we obtain that $w_Y\in A_3$, because $(h_3h_2h_1)^2\in A_3$. 
Also, since $h_3h_2h_1h_2=h_3h_2$ and $x^2$ does not occur in $v_Y$, one of the following conditions hold:\\
\noindent $(z_1\neq y_i)$: $v_Y=h_2h_1$ or $v_Y=h_2$;\\
\noindent $(z_1= y_i)$:] $v_Y=h_3h_2h_1$ or $v_Y=h_3h_2$.\\
Thus, we have $v_Y\in A_2$, a contradiction.
\end{enumerate}
\end{proof}

\begin{lem}\label{x,y,z,t-1}
Let $x,y,z,t\in Y$ with $xy\neq zt$.
Suppose that there exist words $w_1$, $w_2$, $v_1$ and $v_2$ in $Y^*$ 
such that $w_Y=w_1xyw_2$ and $v_Y=v_1ztv_2$.
If one of the following conditions hold:
\begin{enumerate}
\item if $xy$ and $zt$ do not occur in $w_1x$ and $v_1z$, and $x=z$;
\item if $y$ does not occur in $w_1x$ and $v_1z$, and $y=t$;
\item if $xy$ and $zt$ do not occur in $yw_2$ and $tv_2$, and $y=t$;
\item if $x$ does not occur in $yw_2$ and $tv_2$, and $x=z$,
\end{enumerate}
then $\mathcal{J}_5$ does not satisfy the identity $w_Y=v_Y$.
\end{lem}

\begin{proof}
(1) Since $x=z$ and $xy\neq zt$, by symmetry we may assume that
one of the following conditions holds:
\begin{enumerate}
\item [(I)]  $xy=xx$, $zt=xt$ and $x\neq t$;
\item [(II)] $zt=xt$ and $\abs{\{x,y,t\}}=3$.
\end{enumerate}
For every case, we substitute as follows:\\
(I) substitute the letter $x$ the value $h_2h_3h_4$ and the letter $t$ the value $h_1h_2$. If there is another letter in $c(w_Y)\setminus\{x,t\}$, substitute it the value $h_3h_2$. By Lemma~\ref{R_non-eq}.$(\ast2)$, we obtain that $w_Y$ and $v_Y$ are in different $\R$-classes.\\
(II) substitute the letter $x$ the value $h_3h_4$, the letter $y$ the value $h_1h_2$ and the letter $t$ the value $h_2$. If there is a letter in $c(w_Y)\setminus\{x,y,t\}$, substitute it the value $h_3h_2$. Again, by Lemma~\ref{R_non-eq}.$(\ast2)$, we obtain that $w_Y$ and $v_Y$ are in different $\R$-classes.

(2) Substitute the letter $x$ the value $h_3h_4$, the letter $z$ the value $h_4h_3$, the letter $y$ the value $h_2h_3h_1$ and the other letters the value $h_3$. Since $y_1=z_1$, one of the following conditions holds:
\begin{enumerate}
\item [(I)]  $w_1xy=h_3h_4h_2h_3h_1$ and $v_1zy=h_3h_4h_3h_2h_3h_1=h_3h_1$;
\item [(II)] $w_1xy=h_4h_3h_4h_2h_3h_1=h_4h_2h_3h_1$ and $v_1zy=h_4h_3h_2h_3h_1=h_4h_3h_1$.
\end{enumerate}
In both cases, by Lemma~\ref{R_non-eq}.$(\ast1)$, we obtain that $w_Y$ and $v_Y$ are in different $\R$-classes.

(3) As $y=t$ and $xy\neq zt$, by symmetry we may assume that
one of the following conditions holds:
\begin{enumerate}
\item [(I)]  $xy=xx$, $zt=zx$ and $x\neq z$;
\item [(II)] $zt=zy$ and $\abs{\{x,y,z\}}=3$.
\end{enumerate}
Like as above, for every case, we substitute as follows:\\
(I) substitute the letter $x$ the value $h_4h_3h_2$ and the letter $z$ the value $h_2h_1$. If there is another letter in $c(w_Y)\setminus\{x,z\}$, substitute it the value $h_2h_3$. By Lemma~\ref{L_non-eq}.$(\ast2)$, we obtain that $w_Y$ and $v_Y$ are in different $\eL$-classes.\\
(II) substitute the letter $x$ the value $h_2h_1$, the letter $y$ the value $h_4h_3$ and the letter $z$ the value $h_2$. If there is a letter in $c(w_Y)\setminus\{x,y,z\}$, substitute it the value $h_2h_3$. Again, by Lemma~\ref{L_non-eq}.$(\ast2)$, we obtain that $w_Y$ and $v_Y$ are in different $\eL$-classes.

(4) Substitute the letter $y$ the value $h_3h_4$, the letter $t$ the value $h_4h_3$, the letter $x$ the value $h_1h_3h_2$ and the other letters the value $h_3$. Since $y_r=z_s$, one of the following conditions holds:
\begin{enumerate}
\item [(I)]  $xyw_2=h_1h_3h_2h_3h_4h_3=h_1h_3$ and $xtv_2=h_1h_3h_2h_4h_3$;
\item [(II)] $xyw_2=h_1h_3h_2h_3h_4=h_1h_3h_4$ and $xtv_2=h_1h_3h_2h_4h_3h_4=h_1h_3h_2h_4$;
\end{enumerate}
In both cases, by Lemma~\ref{L_non-eq}.$(\ast1)$, we obtain that $w_Y$ and $v_Y$ are in different $\eL$-classes.
\end{proof}

\begin{lem}\label{x,y,z,t-2}
Let $x,y,z,t\in Y$ such that $x\neq z$ and $y\neq t$.
Suppose that there exist words $w_1$, $w_2$, $v_1$ and $v_2$ in $Y^*$ such that
$w_Y=w_1xyw_2$ and $v_Y=v_1ztv_2$. Let $C_1=\{w_1x,v_1z\}$, $C_2=\{yw_2,tv_2\}$ and let $C\in\{C_1,C_2\}$.
If $xy$ and $zt$ do not occur in the elements of $C$, and one of the following states holds:\\
\noindent $\bf{(xt)\colon}$ $xt$ does not occur in the elements of $C$;\\
\noindent $\bf{(zy)\colon}$ $zy$ does not occur in the elements of $C$;\\
\noindent $\bf{(xu,zu')\colon}$ there exists a subset $\{x,y,z,t\}\subseteq Y'\subseteq Y$ such that $Y'=Y_1\cup Y_2$, for some subsets $Y_1$ and $Y_2$, with the following conditions:
\begin{enumerate}
\item $Y_1\cap Y_2=\emptyset$;
\item if $u\in Y_1$, then $xu$ does not occur in $o_{Y'}$, for $o\in C$;
\item if $u'\in Y_2$, then $zu'$ does not occur in $o_{Y'}$, for $o\in C$;
\item $xy$ and $zt$ does not occur in $o_{Y'}$, for $o\in C$,
\end{enumerate}
\noindent $\bf{(uy,u't)\colon}$ there exists a subset $\{x,y,z,t\}\subseteq Y'\subseteq Y$ such that $Y'=Y_1\cup Y_2$, for some subsets $Y_1$ and $Y_2$, with the following conditions:
\begin{enumerate}
\item $Y_1\cap Y_2=\emptyset$;
\item if $u\in Y_1$, then $uy$ does not occur in $o_{Y'}$, for $o\in C$;
\item if $u'\in Y_2$, then $u't$ does not occur in $o_{Y'}$, for $o\in C$;
\item $xy$ and $zt$ does not occur in $o_{Y'}$, for $o\in C$,
\end{enumerate}
then $\mathcal{J}_5$ does not satisfy the identity $w_Y=v_Y$.
\end{lem}

\begin{proof}
Since $x\neq z$ and $y\neq t$, by symmetry, we may assume that one of the following states holds:\\
\noindent $\bf{(xx,zz)\colon}$ $x=y$, $z=t$ and $x\neq z$;\\
\noindent $\bf{(xx,zt)\colon}$ $x=y$ and $\abs{\{x,z,t\}}=3$;\\
\noindent $\bf{(xy,zx)\colon}$ $x=t$ and $\abs{\{x,y,z\}}=3$;\\
\noindent $\bf{(xy,yx)\colon}$ $x=t$, $y=z$ and $x\neq y$;\\
\noindent $\bf{(xy,zt)\colon}$ $\abs{\{x,y,z,t\}}=4$.

For every state $\bf{(xx,zz)}$, $\bf{(xx,zt)}$, $\bf{(xy,zx)}$, $\bf{(xy,yx)}$, $\bf{(xy,zt)}$ and every states of the lemma, we define a homomorphism $\phi\colon c(w_Y)^{\ast}\rightarrow \mathcal{J}_5$ as follows:\\
\noindent $\bf{((xx,zz),(xt))\colon}$ $\phi(x)= h_3h_2h_1$, $\phi(z)= h_4h_3h_2$ and for every $u\in c(w_Y)\setminus \{x,z\}$, $\phi(u)= h_2h_3$.\\
\noindent $\bf{((xx,zz),(zy))\colon}$ $\phi(x)= h_1h_2h_3$, $\phi(z)= h_2h_3h_4$ and for every $u\in c(w_Y)\setminus \{x,z\}$, $\phi(u)= h_3h_2$.\\
\noindent $\bf{((xx,zz),(xu,zu'))\colon}$ $\phi(x)= h_3h_2h_1$, $\phi(z)= h_2h_3h_4$, for every $u\in Y_1\setminus\{x\}$, $\phi(u)= h_3$, for every $u'\in Y_2\setminus\{z\}$,  $\phi(u')= h_2$ and for every $u''\in Y\setminus Y'$, $\phi(u'')= 1$.\\
\noindent $\bf{((xx,zz),(uy,u't))\colon}$ $\phi(x)= h_1h_2h_3$, $\phi(z)= h_4h_3h_2$, for every $u\in Y_1\setminus\{x\}$, $\phi(u)= h_3$, for every $u'\in Y_2\setminus\{z\}$, $\phi(u') = h_2$ and for every $u''\in Y\setminus Y'$, $\phi(u'')= 1$.

\noindent $\bf{((xx,zt),(xt))\colon}$ $\phi(x)= h_3h_2h_1$, $\phi(z)= h_2$, $\phi(t)= h_4h_3$ and for every $u\in c(w_Y)\setminus \{x,z,t\}$, $\phi(u)= h_2h_3$.\\
\noindent $\bf{((xx,zt),(zy))\colon}$ $\phi(x)= h_1h_2h_3$, $\phi(z)= h_3h_4$, $\phi(t)= h_2$ and for every $u\in c(w_Y)\setminus \{x,z,t\}$, $\phi(u)= h_3h_2$.\\
\noindent $\bf{((xx,zt),(xu,zu'))\colon}$ $\phi(x)= h_3h_2h_1$, $\phi(t)= h_2h_3$, for every $u\in Y_1\setminus\{x,z\}$, $\phi(u)= h_3$, for every $u'\in Y_2\setminus\{z,t\}$, $\phi(u')= h_2$ and for every $u''\in Y\setminus Y'$, $\phi(u'')= 1$. If $z\in Y_1$ then $\phi(z)= h_3h_4$, otherwise, $\phi(z)= h_2h_3h_4$.\\
\noindent $\bf{((xx,zt),(uy,u't))\colon}$ $\phi(x)= h_1h_2h_3$, $\phi(z)= h_3h_2$, for every $u\in Y_1\setminus\{x,t\}$, $\phi(u)= h_3$, for every $u'\in Y_2\setminus\{z,t\}$, $\phi(u')= h_2$ and for every $u''\in Y\setminus Y'$, $\phi(u'')= 1$. If $t\in Y_1$ then $\phi(t)= h_4h_3$, otherwise, $\phi(t)= h_4h_3h_2$.

\noindent $\bf{((xy,zx),(xt))\colon}$ $\phi(x)= h_4h_3h_2h_1$, $\phi(y)= h_3$, $\phi(z)= h_2$ and for every $u\in c(w_Y)\setminus \{x,y,z\}$, $\phi(u)= h_2h_3$.\\
\noindent $\bf{((xy,zx),(zy))\colon}$ $\phi(x)= h_2h_3$, $\phi(y)= h_1h_2$, $\phi(z)= h_3h_4$ and for every $u\in c(w_Y)\setminus \{x,y,z\}$, $\phi(u)= h_3h_2$.\\
\noindent $\bf{((xy,zx),(xu,zu'))\colon}$ $\phi(x)= h_2h_1$, $\phi(y)= h_3$, for every $u\in Y_1\setminus\{z\}$, $\phi(u)= h_3$, for every $u'\in Y_2\setminus\{z,x\}$, $\phi(u')= h_2$ and for every $u''\in Y\setminus Y'$, $\phi(u'')= 1$. If $z\in Y_1$ then $\phi(z)= h_3h_4$, otherwise, $\phi(z)= h_2h_3h_4$.\\
\noindent $\bf{((xy,zx),(uy,u't))\colon}$ $\phi(x)= h_4h_3$, $\phi(z)= h_3h_2$, for every $u\in Y_1\setminus\{y,x\}$, $\phi(u)= h_3$, for every $u'\in Y_2\setminus\{y,z\}$, $\phi(u')= h_2$ and for every $u''\in Y\setminus Y'$, $\phi(u'')= 1$. If $y\in Y_1$ then $\phi(y)= h_1h_2h_3$, otherwise, $\phi(y)=h_1h_2$.

\noindent $\bf{((xy,yx),(xt))\colon}$ $\phi(x)= h_4h_3h_2h_1$, $\phi(y)= h_3h_2$ and for every $u\in c(w_Y)\setminus \{x,y\}$, $\phi(u)= h_2h_3$.\\
\noindent $\bf{((xy,yx),(zy))\colon}$ $\phi(x)= h_2h_3$, $\phi(y)= h_1h_2h_3h_4$ and for every $u\in c(w_Y)\setminus \{x,y\}$, $\phi(u)= h_3h_2$.\\
\noindent $\bf{((xy,yx),(xu,zu'))\colon}$  $\phi(x)= h_2h_1$, $\phi(y)= h_3h_4$, for every $u\in Y_1\setminus\{y\}$, $\phi(u)= h_3$, for every $u'\in Y_2\setminus\{x\}$,  $\phi(u')= h_2$ and for every $u''\in Y\setminus Y'$, $\phi(u'')= 1$.\\
\noindent $\bf{((xy,yx),(uy,u't))\colon}$ $\phi(x)= h_4h_3$, $\phi(y)= h_1h_2$, for every $u\in Y_1\setminus\{x\}$, $\phi(u)= h_3$, for every $u'\in Y_2\setminus\{y\}$,  $\phi(u')= h_2$ and for every $u''\in Y\setminus Y'$, $\phi(u'')= 1$.

\noindent $\bf{((xy,zt),(xt))\colon}$ $\phi(x)= h_2h_1$, $\phi(y)= h_3$, $\phi(z)= h_2$, $\phi(t)= h_4h_3$ and for every $u\in c(w_Y)\setminus \{x,y,z,t\}$, $\phi(u)= h_2h_3$.\\
\noindent $\bf{((xy,zt),(zy))\colon}$ $\phi(x)= h_3$, $\phi(y)= h_1h_2$, $\phi(z)= h_3h_4$, $\phi(t)= h_2$ and for every $u\in c(w_Y)\setminus \{x,y,z,t\}$, $\phi(u)= h_3h_2$.\\
\noindent $\bf{((xy,zt),(xu,zu'))\colon}$ $\phi(y)= h_3$, $\phi(t)= h_2$, for every $u\in Y_1\setminus \{x,z\}$, $\phi(u)= h_3$, for every $u'\in Y_2 \setminus \{x,z\}$,  $\phi(u')= h_2$ and for every $u''\in Y\setminus Y'$, $\phi(u'')= 1$.  If $x\in Y_1$ then $\phi(x)= h_3h_2h_1$, otherwise, $\phi(x)= h_2h_1$. Also, if $z\in Y_1$ then $\phi(z)= h_3h_4$, otherwise, $\phi(z)= h_2h_3h_4$.\\ 
\noindent $\bf{((xy,zt),(uy,u't))\colon}$ $\phi(x)= h_2h_3$, $\phi(z)= h_3h_2$, for every $u\in Y_1\setminus \{x,y,t\}$, $\phi(u)= h_3$, for every $u'\in Y_2 \setminus \{y,z,t\}$,  $\phi(u')= h_2$ and for every $u''\in Y\setminus Y'$, $\phi(u'')= 1$.  If $y\in Y_1$ then $\phi(y)= h_1h_2h_3$, otherwise, $\phi(y)= h_1h_2$. Also, if $t\in Y_1$ then $\phi(t)= h_4h_3$, otherwise, $\phi(t)= h_4h_3h_2$.

By the substitutions for the state $\bf{(xx,zz)}$ subject to the conditions of the states $\bf{(xt)}$, $\bf{(zy)}$, $\bf{(xu,zu')}$ and $\bf{(uy,u't)}$, the elements $\phi(xx)$ and $\phi(zz)$ are in $A_3$, and the elements $\phi(w_1x)$ and $\phi(v_1z)$ are in $A_2$, for $C=C_1$. Also, for $C=C_2$, the elements $\phi(yw_2)$ and $\phi(tv_2)$ are in $A_2$. These substitutions satisfy the conditions of Lemma~\ref{R_non-eq}.$(\ast1)$ or Lemma~\ref{L_non-eq}.$(\ast1)$. Therefore, $\mathcal{J}_5$ does not satisfy the identity $w_Y=v_Y$, for the state $\bf{(xx,zz)}$ and $C\in\{C_1,C_2\}$. We have same result for other states $\bf{(xx,zt)}$, $\bf{(xy,zx)}$, $\bf{(xy,yx)}$ and $\bf{(xy,zt)}$ by others substitutions as above. Hence, we conclude that $\mathcal{J}_5$ does not satisfy the identity $w_Y=v_Y$.
\end{proof}

\begin{thm}\label{main}
Let $w$ and $v$ be words of $X^*$. The monoid $\mathcal{J}_5$ satisfies the identity $w=v$ if and only if $c(w)=c(v)$ and for every non empty subset $Y$ of $c(w)$, the following conditions hold:
\begin{enumerate}
\item the first letter of $w_Y$ and $v_Y$ are equal;
\item the last letter of $w_Y$ and $v_Y$ are equal;
\item each word of length 2 occurs in $w_Y$ if and only if it occurs in $v_Y$;
\item let $x,y,z,t\in Y$ with $xy\neq zt$.
Suppose that there exist words $w_1$, $w_2$, $v_1$ and $v_2$ in $Y^*$ 
such that $w_Y=w_1xyw_2$ and $v_Y=v_1ztv_2$. Let $C_1=\{w_1x,v_1z\}$, $C_2=\{yw_2,tv_2\}$ and let $C\in\{C_1,C_2\}$.
The following conditions hold:
\begin{enumerate}
\item if $xy$ and $zt$ do not occur in the elements of $C_1$, then $x\neq z$;
\item if $y$ does not occur in the elements of $C_1$, then $y\neq t$;
\item if $xy$ and $zt$ do not occur in the elements of $C_2$, then $y\neq t$;
\item if $x$ does not occur in the elements of $C_2$, then $x\neq z$;
\item if $x\neq z$, $y\neq t$ and, $xy$ and $zt$ do not occur in the elements of $C$ then the following conditions hold;
\begin{enumerate}
\item $xt$ occurs in the elements of $C$;
\item $zy$ occurs in the elements of $C$;
\item there does not exist a subset $\{x,y,z,t\}\subseteq Y'\subseteq Y$ such that $Y'=Y_1\cup Y_2$, for some subsets $Y_1$ and $Y_2$, with the following conditions:
\begin{enumerate}
\item $Y_1\cap Y_2=\emptyset$;
\item if $u\in Y_1$, then $xu$ does not occur in $o_{Y'}$, for $o\in C$;
\item if $u'\in Y_2$, then $zu'$ does not occur in $o_{Y'}$, for $o\in C$;
\item $xy$ and $zt$ does not occur in $o_{Y'}$, for $o\in C$,
\end{enumerate}
\item there does not exist a subset $\{x,y,z,t\}\subseteq Y'\subseteq Y$ such that $Y'=Y_1\cup Y_2$, for some subsets $Y_1$ and $Y_2$, with the following conditions:
\begin{enumerate}
\item $Y_1\cap Y_2=\emptyset$;
\item if $u\in Y_1$, then $uy$ does not occur in $o_{Y'}$, for $o\in C$;
\item if $u'\in Y_2$, then $u't$ does not occur in $o_{Y'}$, for $o\in C$;
\item $xy$ and $zt$ does not occur in $o_{Y'}$, for $o\in C$.
\end{enumerate}
\end{enumerate}
\end{enumerate}
\end{enumerate}
\end{thm}

\begin{proof}
If $\mathcal{J}_5$ satisfies the identity $w=v$, then by Lemmas~\ref{c-w-v}, we have $c(w)=c(v)$ and $\mathcal{J}_5$ satisfies $w_Y=v_Y$, for every subset $Y\subseteq c(w)$. It is followed easily by substituting every letter in $X\setminus Y$ value 1. Then, by~\ref{M-x_1=y_1-x_n=y_m}, \ref{word-of-length-2}, \ref{x,y,z,t-1} and \ref{x,y,z,t-2}, all Conditions~(1), (2), (3) and (4) hold.

Now, suppose the contrary that $\mathcal{J}_5$ does not satisfy the identity $w=v$, $c(w)=c(v)$ and the conditions of the theorem hold. 
Hence, there exists a homomorphism $\phi\colon c(w)^{\ast}\rightarrow \mathcal{J}_5$ such that $\phi(w)\neq \phi(v)$.
Let $Y=\{y\in c(w)\mid \phi(y)\neq 1\}$. It easily follows that $\phi(w_Y)\neq \phi(v_Y)$.
There exist letters $y_1, \ldots, y_r, z_1, \ldots, z_s$ in $Y$ 
such that  $w_Y= y_1\cdots y_r$ and $v_Y= z_1\cdots z_s$. By Condition~(3), we have $\phi(w_Y),\phi(v_Y)\in A_2$ or $\phi(w_Y),\phi(v_Y)\in A_3$. By Conditions~(1) and (2), we have $y_1=z_1$ and $y_r=z_s$. Now, as $\mathcal{J}_5$ is aperiodic, if $\phi(w_Y),\phi(v_Y)\in A_2$, then we have $\phi(w_Y)=\phi(v_Y)$, a contradiction. Hence, we have $\phi(w_Y),\phi(v_Y)\in A_3$. Since $\phi(w_Y)\neq \phi(v_Y)$ and $\mathcal{J}_5$ is aperiodic, one or both of the following conditions holds:
\begin{enumerate}
\item $\phi(w_Y)$ and $\phi(v_Y)$ are not in the same $\R$-class;
\item $\phi(w_Y)$ and $\phi(v_Y)$ are not in the same $\eL$-class.
\end{enumerate}

First, we suppose that $\phi(w_Y)$ and $\phi(v_Y)$ are not in the same $\R$-class. Since $y_1=z_1$, we have $\phi(y_1)=\phi(z_1)$. If $\phi(y_1)\in A_3$, then $\phi(w_Y)$ and $\phi(v_Y)$ are in the same $\R$-class. Hence, we have $\phi(y_1)\in A_2$. As $\phi(w_Y),\phi(v_Y)\in A_3$, there exist integers $1\leq i_1< r$ and $1\leq i_2< s$ such that $\phi(y_1\cdots y_{i_1})\in A_2$, $\phi(y_1\cdots y_{i_1+1})\in A_3$, $\phi(z_1\cdots z_{i_2})\in A_2$ and $\phi(z_1\cdots z_{i_2+1})\in A_3$. If $y_{i_1}y_{i_1+1}=z_{i_2}z_{i_2+1}$, then we have $y_{i_1}=z_{i_2}$. It follows that $\phi(y_1\cdots y_{i_1})=\phi(z_1\cdots z_{i_2})$, because $\mathcal{J}_5$ is aperiodic, $\phi(y_1\cdots y_{i_1}),\phi(z_1\cdots z_{i_2})\in A_2$ and $y_1=z_1$. Now, as $y_{i_1+1}=z_{i_2+1}$, $\phi(w_Y)$ and $\phi(v_Y)$ are in the same $\R$-class, a contradiction. Then, we have $y_{i_1}y_{i_1+1}\neq z_{i_2}z_{i_2+1}$.

If $\phi(y_{i_1+1})\in A_3$, then $y_{i_1+1}$ does not occur in $y_1\cdots y_{i_1}$ and $z_1\cdots z_{i_2}$ and by Condition~(4).(b), we have $y_{i_1+1}\neq z_{i_2+1}$.
If $z_{i_2+1}$ does not occur in $y_1\cdots y_{i_1}$ and $z_1\cdots z_{i_2}$, 
then the words $w_{\{y_{i_1+1},z_{i_2+1}\}}$ and $v_{\{y_{i_1+1},z_{i_2+1}\}}$ do not satisfy Condition~(1). 
Hence, $z_{i_2+1}$ occurs in one or both words $y_1\cdots y_{i_1}$ and $z_1\cdots z_{i_2}$ that it causes $\phi(z_{i_2+1})\in A_2$.
Since $\phi(z_1\cdots z_{i_2+1})\in A_3$ and $\phi(z_{i_2+1})\in A_2$, 
we have $\phi(z_{i_2}z_{i_2+1})\in A_3$ and, thus,
the word $z_{i_2}z_{i_2+1}$ does not occur in $y_1\cdots y_{i_1}$ and $z_1\cdots z_{i_2}$. 
Now, as $y_{i_1+1}$ does not occur in $y_1\cdots y_{i_1}$ and $z_1\cdots z_{i_2}$, by Condition~(4).(a), we have $y_{i_1}\neq z_{i_2}$.
Now, as $y_{i_1+1}\neq z_{i_2+1}$, Condition~(4).(e).(i) or (4).(e).(ii) does not hold, a contradiction. 
Then, we have $\phi(y_{i_1+1})\in A_2$. Similarly, we have $\phi(z_{i_2+1})\in A_2$.

Now, as $\phi(y_1\cdots y_{i_1+1}), \phi(z_1\cdots z_{i_2+1})\in A_3$ and $\phi(y_1\cdots y_{i_1})$, $\phi(y_{i_1+1})$, $\phi(z_1\cdots z_{i_2})$, $\phi(z_{i_2+1})\in A_2$, we have $\phi(y_{i_1}y_{i_1+1}), \phi(z_{i_2}z_{i_2+1})\in A_3$. 
Hence, we have 
\begin{align*}
\phi(y_{i_1}y_{i_1+1}), \phi(z_{i_2}z_{i_2+1})\in \{
&\alpha_1h_1h_3\alpha_2,\beta_1h_1h_4\beta_2,\gamma_1h_2h_4\gamma_2,\alpha'_1h_3h_1\alpha'_2,\\
&\beta'_1h_4h_1\beta'_2,\gamma'_1h_4h_2\gamma'_2\mid \text{for some elements}\ \alpha_1,\alpha_2,\\
&\beta_1,\beta_2, \gamma_1,\gamma_2,
\alpha'_1,\alpha'_2,\beta'_1,\beta'_2,\gamma'_1,\gamma'_2\in A_2\cup A_1\\
&\text{for which}\ \alpha_1h_1,h_3\alpha_2,\beta_1h_1,h_4\beta_2,\gamma_1h_2,\\
&h_4\gamma_2,\alpha'_1h_3,h_1\alpha'_2, \beta'_1h_4,h_1\beta'_2,\gamma'_1h_4,h_2\gamma'_2\in A_2\}.
\end{align*}
By Lemma~\ref{R_non-eq} and by symmetry, we may assume that, there exist some elements $\alpha,\beta,\gamma,\lambda\in A_2\cup A_1$ such that one of the following conditions holds:
\begin{enumerate}
\item [(A1)] $\phi(y_{i_1})=\alpha a$, $\phi(y_{i_1+1})=b\beta$, $\phi(z_{i_2})=\gamma c$ and $\phi(z_{i_2+1})=d\lambda$, for some elements $\{a,b\}=\{h_1,h_3\}$ and $\{c,d\}=\{h_2,h_4\}$;
\item [(A2)] $\phi(y_{i_1})=\alpha h_1$, $\phi(y_{i_1+1})=h_4\beta$, $\phi(z_{i_2})=\gamma h_1$ and $\phi(z_{i_2+1})=h_3\lambda$;
\item [(A3)] $\phi(y_{i_1})=\alpha h_1$, $\phi(y_{i_1+1})=h_4\beta$, $\phi(z_{i_2})=\gamma h_3$ and $\phi(z_{i_2+1})=h_1\lambda$;
\item [(A4)] $\phi(y_{i_1})=\alpha h_4$, $\phi(y_{i_1+1})=h_1\beta$, $\phi(z_{i_2})=\gamma h_2$ and $\phi(z_{i_2+1})=h_4\lambda$;
\item [(A5)] $\phi(y_{i_1})=\alpha h_4$, $\phi(y_{i_1+1})=h_1\beta$, $\phi(z_{i_2})=\gamma h_4$ and $\phi(z_{i_2+1})=h_2\lambda$.
\end{enumerate}
We obtain that $y_{i_1}y_{i_1+1}\neq z_{i_2}z_{i_2+1}$ and as $\phi(y_{i_1}y_{i_1+1}), \phi(z_{i_2}z_{i_2+1})\in A_3$, the words 
$y_{i_1}y_{i_1+1}$ and $z_{i_2}z_{i_2+1}$ do not occur in the words $y_1\cdots y_{i_1+1}$ and $z_1\cdots z_{i_2+1}$. By Condition~(4).(a), we have $y_{i_1}\neq z_{i_2}$. Also, by considering Conditions~(A1), (A2), (A3), (A4) and (A5), we have $y_{i_1+1}\neq z_{i_2+1}$.

First, suppose that Condition~(A1) holds. 
If $a=h_1$, then by Condition~(4).(e).(i), we have $d=h_2$. Then, we have $b=h_3$ and $c=h_4$. Also, if $a=h_3$, then we have $b=h_1$ and by Condition~(4).(e).(ii), we have $c=h_2$. Then, we have $d=h_4$. Then, we have $(a,c)=(h_1,h_4)$ or $(b,d)=(h_1,h_4)$.
By Condition~(4).(e).(iii), there exists a letter $u_1\in Y$ such that $y_{i_1}u_1$ occurs in $y_1\cdots y_{i_1}$ or $z_1\cdots z_{i_2}$ and also $z_{i_2}u_1$ occurs in $y_1\cdots y_{i_1}$ or $z_1\cdots z_{i_2}$. 
Hence, we have $\phi(y_{i_1}u_1),\phi(z_{i_2}u_1)\in A_2$.
If $a=h_1$ and $c=h_4$, then we have $y_{i_1}u_1\in A_3$ or $z_{i_2}u_1\in A_3$, a contradiction. Because, there does not exist an element $h\in A_2$ such that $h_1h,h_4h\in A_2$. Also, by Condition~(4).(e).(iv), there exists a letter $u_2\in Y$ such that $u_2y_{i_1+1}$ occurs in $y_1\cdots y_{i_1}$ or $z_1\cdots z_{i_2}$ and also $u_2z_{i_2+1}$ occurs in $y_1\cdots y_{i_1}$ or $z_1\cdots z_{i_2}$. 
Then, we have $\phi(u_2y_{i_1+1}),\phi(u_2z_{i_2+1})\in A_2$.
If $b=h_1$ and $d=h_4$, then we have $u_2y_{i_1+1}\in A_3$ or $u_2z_{i_2+1}\in A_3$, a contradiction. 
Therefore, Condition~(A1) does not hold. 
Also, if Conditions~(A2) or (A5) holds, then Condition~(4).(e).(i) fails.
Similarly, if Condition~(A3) or (A4) holds, then Condition~(4).(e).(iv) fails.
Therefore, $\phi(w_Y)$ and $\phi(v_Y)$ are in the same $\R$-class.

Similarly, by Lemma~\ref{L_non-eq}, $\phi(w_Y)$ and $\phi(v_Y)$ are in the same $\eL$-class and, thus, $\mathcal{J}_5$ satisfies the identity $w=v$.
\end{proof}

By Theorem~\ref{main}, the following corollary easily follows.

\begin{cor}\label{x3-x2}
The monoid $\mathcal{J}_5$ satisfies the identities $x^3=x^2$ and does not satisfy the identity $x^2=x$.
Also, if $\mathcal{J}_5$ satisfies the identity $w=v$, then we have $\abs{w_{x}}\geq 2$ if and only if $\abs{v_{x}}\geq 2$, for every $x\in X$.
\end{cor}

Note that, in Theorem~\ref{main}, if $xy$ and $zt$ do not occur in $w_1x$ and $v_1z$, and $y=t$, for a subset $Y\subseteq c(w)$, $\mathcal{J}_5$ may satisfy the identity $w=v$. Also, if $xy$ and $zt$ do not occur in $yw_2$ and $tv_2$, and $x=z$, $\mathcal{J}_5$ may satisfy the identity $w=v$. For example, by Theorem~\ref{main}, $\mathcal{J}_5$ satisfies the following identities:
\begin{align*}
xu^2xuz^2u \bb{xx} u^2z^2x^2&=xu^2xuz^2u \bb{zx} u^2z^2x^2,\\
y^2z^2x^2z \bb{xy} xzyz^2x^2y^2&=y^2z^2x^2z \bb{zy} xzyz^2x^2y^2,\\
x^2t^2u^2 \bb{xx} ut^2uxu^2x&=x^2t^2u^2 \bb{xt} ut^2uxu^2x,\\
x^2y^2t^2xty \bb{xy} ty^2t^2x^2&=x^2y^2t^2xty \bb{xt}ty^2t^2x^2.
\end{align*}
Also, note that, there exists an identity such that $xy\neq zt$, $x\neq z$, $y\neq t$ and, $xy$ and $zt$ do not occur in the elements of $C_1$ or $C_2$. For example, for $C_1$, we have the following identity that satisfies all conditions of Theorem~\ref{main}:
\begin{align*}
yxyxzxyzyz\bb{xx}zxzyz^2y^2x^2=yxyxzxyzyz\bb{yy}zxzyz^2y^2x^2.
\end{align*}

In~\cite{Kit-Vol}, the authors characterize the identities of $\mathcal{J}_4$. They show that for words $w$ and $v$ of $X^*$, the monoid $\mathcal{J}_4$ satisfies the identity $w=v$ if $c(w)=c(v)$ and for every non empty subset $Y$ of $c(w)$, Conditions~(1), (2) and (3) of Theorem~\ref{main} hold. Hence, Condition~(4) of Theorem~\ref{main} has a key role to recognize the identities of $\mathcal{J}_4$ which $\mathcal{J}_5$ does not satisfy them.
 

\section*{Acknowledgments}
The author was partially supported by CMUP, which is financed by
national funds through FCT -- Fundação para a Ciência e a Tecnologia,
I.P., under the project UIDB/00144/2020. The author also
acknowledges FCT support through a contract based on the “Lei do
Emprego Científico” (DL 57/2016).

\bibliographystyle{plain}
\bibliography{ref-J5}

\def\cprime{$'$}
\begin{thebibliography}{10}

\bibitem{Alm}
J.~Almeida.
\newblock {\em Finite semigroups and universal algebra}, volume~3 of {\em
  Series in Algebra}.
\newblock World Scientific Publishing Co., Inc., River Edge, NJ, 1994.
\newblock Translated from the 1992 Portuguese original and revised by the
  author.

\bibitem{Aui-Vol}
K.~Auinger, Yuzhu Chen, Xun Hu, Yanfeng Luo, and M.~V. Volkov.
\newblock The finite basis problem for {K}auffman monoids.
\newblock {\em Algebra Universalis}, 74(3-4):333--350, 2015.

\bibitem{Bor-Mir-Do-Pet}
M.~Borisavljevi\'{c}, K.~Do\v{s}en, and Z.~Petri\'{c}.
\newblock Kauffman monoids.
\newblock {\em J. Knot Theory Ramifications}, 11(2):127--143, 2002.

\bibitem{Bra}
R.~Brauer.
\newblock On algebras which are connected with the semisimple continuous
  groups.
\newblock {\em Ann. of Math. (2)}, 38(4):857--872, 1937.

\bibitem{Che-Hu-Kit-Vol}
Yuzhu Chen, Xun Hu, N.~V. Kitov, Yanfeng Luo, and M.~V. Volkov.
\newblock Identities of the {K}auffman monoid {$\mathcal{K}_3$}.
\newblock {\em Comm. Algebra}, 48(5):1956--1968, 2020.

\bibitem{Cli-Pre}
A.~H. Clifford and G.~B. Preston.
\newblock {\em The algebraic theory of semigroups. {V}ol. {I}}.
\newblock Mathematical Surveys, No. 7. American Mathematical Society,
  Providence, R.I., 1961.

\bibitem{Gre}
J.~A. Green.
\newblock On the structure of semigroups.
\newblock {\em Ann. of Math. (2)}, 54:163--172, 1951.

\bibitem{Kau}
L.~H. Kauffman.
\newblock An invariant of regular isotopy.
\newblock {\em Trans. Amer. Math. Soc.}, 318(2):417--471, 1990.

\bibitem{Kit-Vol}
N.~V. Kitov and M.~V. Volkov.
\newblock Identities of the {K}auffman {M}onoid {$\mathcal{K}_4$} and of the
  {J}ones {M}onoid {$\mathcal{J}_4$}.
\newblock In {\em Fields of logic and computation. {III}}, volume 12180 of {\em
  Lecture Notes in Comput. Sci.}, pages 156--178. Springer, Cham, [2020]
  \copyright 2020.

\bibitem{Lau-Kwo-Fit}
Kwok~Wai Lau and D.~G. FitzGerald.
\newblock Ideal structure of the {K}auffman and related monoids.
\newblock {\em Comm. Algebra}, 34(7):2617--2629, 2006.

\bibitem{Rho-Ste}
J.~Rhodes and B.~Steinberg.
\newblock {\em The {$q$}-theory of finite semigroups}.
\newblock Springer Monographs in Mathematics. Springer, New York, 2009.

\end{thebibliography}


\end{document}